\pdfoutput=1 

\documentclass{amsart}

\usepackage{amsmath}%
\usepackage{amsfonts}%
\usepackage{amssymb}%
\usepackage{graphicx}
\usepackage{bm}
\usepackage{MnSymbol}
\usepackage[all,cmtip]{xy}
\usepackage{cite}
\usepackage{hyperref}
\usepackage{tensor}
\usepackage{tikz}
\usepackage{rotating}
\usepackage{float}
\usepackage{tqft}

\newtheorem{dummy}{dummy}[section]              
\newtheorem{lemma}[dummy]{Lemma}
\newtheorem{theorem}[dummy]{Theorem}

\newtheorem{proposition}[dummy]{Proposition}
\theoremstyle{definition}                                  
\newtheorem{definition}[dummy]{Definition}
\newtheorem{example}[dummy]{Example}
\newtheorem{remark}[dummy]{Remark}
\newtheorem{notation}[dummy]{Notation}

\newcommand{\Bord}{\mathbf{Bord}}
\newcommand{\Cob}{\mathbf{Cob}}
\newcommand{\Alg}{\mathbf{Alg}}
\newcommand{\Vect}{\mathbf{Vect}}
\newcommand{\Fam }{\mathbf{Fam}}
\newcommand{\Aut }{\mathrm{Aut}}

\newcommand{\Cov}{\mathcal Cov}
\newcommand{\Hom}{\mathrm{Hom}}
\newcommand{\HOM}{\mathrm{HOM}}
\newcommand{\End}{\mathrm{End}}
\newcommand{\CUP}{\mathbf{Cup}}
\newcommand{\CAP}{\mathbf{Cap}}
\newcommand{\PANTS}{\mathbf{Pants}}
\newcommand{\SADDLE}{\mathbf{Saddle}}

\newcommand{\CHAPS}{\mathbf{Chaps}}
\newcommand{\WHISTLE}{\mathbf{Whistle}}

\newcommand{\too}{\longrightarrow}

\newcommand{\Sum}{\mathrm{Sum}}

\newcommand{\colim }{\mathrm{colim}}
\newcommand{\ZZ}{\mathbb Z}

\newcommand{\SP}{\mathrm{SP}}
\newcommand{\OP}{\mathrm{OP}}

\newcommand{\Ab}{\mathcal A b}

\newcommand{\KO}{ KO}

\newcommand{\Cl}{\mathcal C \! \ell}
\newcommand{\bmu}{{\bm{\mu}}}
\newcommand{\bnu}{{\bm{\nu}}}

\newcommand{\END}{\mathrm{End}}
\newcommand{\PP}{\mathrm{P}}

\newcommand{\lrsub}[3]{\tensor[_{#1}]{{#2}}{_{#3}}}

\newcommand{\cC}{\mathcal C}
\newcommand{\cD}{\mathcal D}

\newcommand{\cH}{\mathcal H}

\newcommand{\cS}{\mathcal S}
\newcommand{\cT}{\mathcal T}

\newcommand{\cY}{\mathcal Y}
\newcommand{\cZ}{\mathcal Z}

\newcommand{\C}{\mathbb C}
\newcommand{\F}{\mathbb F}

\newcommand{\R}{\mathbb R}
\newcommand{\Z}{\mathbb Z}

\newcommand{\sidewayssimeq}{
     \begin{sideways}$\simeq$\end{sideways}}

\newcommand{\tqftid}{

\draw (0,0) -- (2,0);
\draw[fill] (0,0) circle [radius=0.05];
\draw node [above] at (0,0) {$+$};
\draw[fill] (2,0) circle [radius=0.05];
\draw node [above] at (2,0) {$+$};
} 

\newcommand{\tqftidminus}{

\draw (0,0) -- (2,0);
\draw[fill] (0,0) circle [radius=0.05];
\draw node [above] at (0,0) {$-$};
\draw[fill] (2,0) circle [radius=0.05];
\draw node [above] at (2,0) {$-$};
} 

\newcommand{\tqftidpm}{
\draw (0,0) -- (2,0);
\draw[fill] (0,0) circle [radius=0.05];
\draw node [above] at (0,0) {$\pm$};
\draw[fill] (2,0) circle [radius=0.05];
\draw node [above] at (2,0) {$\pm$};
}

\newcommand{\tqftev}{

\draw (0,0) arc [radius=1, start angle=90, end angle=-90];
\draw[fill] (0,0) circle [radius=0.05];
\draw node [above] at (0,0) {$-$};
\draw node [above] at (0,-2) {$+$};
\draw[fill] (0,-2) circle [radius=0.05];
}

\newcommand{\tqftcoev}{

\draw (0,0) arc [radius=1, start angle=-270, end angle=-90];
\draw[fill] (0,0) circle [radius=0.05];
\draw node [above] at (0,0) {$+$};
\draw node [above] at (0,-2) {$-$};
\draw[fill] (0,-2) circle [radius=0.05];
}

\newcommand{\tqfts}{

\draw (0,0) to [out=0, in =180] (2,-2);
\draw (0,-2) to [out=0, in =180] (2,0);
\draw[fill] (0,0) circle [radius=0.05];
\draw node [above] at (0,0) {$+$};
\draw[fill] (2,0) circle [radius=0.05];
\draw node [above] at (2,0) {$-$};
\draw[fill] (0,-2) circle [radius=0.05];
\draw node [above] at (0,-2) {$-$};
\draw[fill] (2,-2) circle [radius=0.05];
\draw node [above] at (2,-2) {$+$};
}

\newcommand{\tqftalphapm}{
\draw (0,0) to [out=0,in=180] (0.9,0) to [out=0, in=0] (1,.3) to [out=180, in=180] (1.1,0) to [out=0,in=180] (2,0);
\draw[fill] (0,0) circle [radius=0.05];
\draw node [above] at (0,0) {$\pm$};
\draw[fill] (2,0) circle [radius=0.05];
\draw node [above] at (2,0) {$\pm$};
}

\newcommand{\tqftevbar}{
\begin{scope}[yshift=-2cm]
\draw (0,0) to [out=0, in=-90] (.7, .8) to [out=-270, in=-270] (1,1) to [out=-90, in=-90] (.7, 1.2) to [out=-270,in=0] (0,2);
\draw[fill] (0,0) circle [radius=0.05];
\draw node [above] at (0,0) {$+$};
\draw node [above] at (0,2) {$-$};
\draw[fill] (0,2) circle [radius=0.05];
\end{scope}
}

\newcommand{\tqftcoevbar}{
\draw (0,0) to [out=180, in=90] (-.7,- .8) to [out=270, in=270] (-1,-1) to [out=90, in=90] (-.7, -1.2) to [out=270,in=180] (0,-2);
\draw[fill] (0,0) circle [radius=0.05];
\draw node [above] at (0,0) {$+$};
\draw node [above] at (0,-2) {$-$};
\draw[fill] (0,-2) circle [radius=0.05];
}

\newcommand{\tqftsper}{
\draw (0.5,0) to [out=0, in=135] (1,-1) to [out=315, in=180] (1.5,-2) to [out=0, in=270] (2,-1) to [out=90, in=0] (1.5,0) to [out=180, in=45] (1,-1) to [out=225,in=0] (0.5,-2) to [out=180, in=270] (0,-1) to [out=90,in=180] (0.5,0);
}

\newcommand{\tqftdisc}{
\draw (1.618,.6266) to [out=90, in=0] (.5,2) to [out=180,in=90] (-0.618,.2394);

\draw [dashed, domain=20:200, smooth, variable=\t]
plot ({cos(\t) + 0.5*sin(\t)+.5},{0.433* sin(\t) +0.433});
\draw [domain=200:380, smooth, variable=\t]
plot ({cos(\t) + 0.5*sin(\t)+.5},{0.433* sin(\t) +0.433});

\begin{scope}[xshift=.5cm,yshift=.433cm]
\draw[fill] (0.5,.433) circle [radius=0.05];
\draw node [above left] at (0.5,0.433) {$-$};
\draw[fill] (-.5, -.433) circle [radius=0.05];
\draw node [above left] at (-.5,-.433) {$+$};
\end{scope}

}

\newcommand{\tqftsaddle}{
\begin{scope}[yscale=1, cm={1,0,.5,0.433,(0,0)}]
\draw (0,1) arc (90:270:1);
\draw (-2.5,-1) arc (270:450:1);
\end{scope}
\draw (-1.382,.1936) to [out=270,in=180] (-1.25,-.3) to [out=0,in=270] (-1.118,-0.2606);

\draw (.5,.433) -- (0.5,-1.567);
\draw (-.5,-.433) -- (-.5, -2.433);
\draw (-2,.433) -- (-2,-0.2);
\draw [dashed] (-2,-.2) -- (-2,-1.567);
\draw (-3, -.433) -- (-3,-2.433);

\draw [dashed](-2,-1.567) -- (-0.5,-1.567);
\draw (-.5,-1.567) -- (.5, -1.567);
\draw (-3, -2.433) -- (-.5, -2.433);

\draw[fill] (0.5,.433) circle [radius=0.05];
\draw node [above left] at (0.5,0.433) {$-$};
\draw[fill] (-.5, -.433) circle [radius=0.05];
\draw node [above left] at (-.5,-.433) {$+$};
\draw [fill] (-2,.433) circle [radius=.05];
\draw node [above left] at (-2,.433) {$-$};
\draw [fill] (-3,-.433) circle [radius=.05];
\draw node [above left] at (-3,-.433) {$+$};

\draw[fill] (0.5,-1.567) circle [radius=0.05];
\draw node [above left] at (0.5,-1.567) {$-$};
\draw[fill] (-.5, -2.433) circle [radius=0.05];
\draw node [above left] at (-.5,-2.433) {$+$};
\draw [fill] (-2,-1.567) circle [radius=.05];
\draw node [above left] at (-2,-1.567) {$-$};
\draw [fill] (-3,-2.433) circle [radius=.05];
\draw node [above left] at (-3,-2.433) {$+$};

}

\newcommand{\tqftadjunction}{
\begin{scope}[yscale=1, cm={1,0,.5,0.433,(0,0)}]
\draw (-2.5,-1) arc (270:450:1);
\end{scope}

\begin{scope}[yshift=2cm, cm={1,0,.5,0.433,(0,0)}]
\draw (-2.5,-1) arc (270:450:1);
\end{scope}

\draw (-1.382,.1936) to [out=270,in=180] (-1.25,-.3) to [out=0,in=270] (-1.118,-0.2606);

\draw [dashed] (.5,.433) -- (0.5,-1.567);
\draw (-.5,-.433) -- (-.5, -2.433);
\draw [dashed] (-2, 1.8) -- (-2,.433) -- (-2,-1.567);
\draw (-2,2.433) -- (-2,1.8);
\draw (-3,1.567)-- (-3, -.433) -- (-3,-2.433);

\draw  [dashed] (-2,-1.567) -- (0.5,-1.567);
\draw (-3, -2.433) -- (-.5, -2.433);

\draw[fill] (0.5,.433) circle [radius=0.05];
\draw node [above left] at (0.5,0.433) {$-$};
\draw[fill] (-.5, -.433) circle [radius=0.05];
\draw node [above left] at (-.5,-.433) {$+$};
\draw [fill] (-2,.433) circle [radius=.05];
\draw node [above left] at (-2,.433) {$-$};
\draw [fill] (-3,-.433) circle [radius=.05];
\draw node [above left] at (-3,-.433) {$+$};

\draw [fill] (-2,.433+2) circle [radius=.05];
\draw node [above left] at (-2,.433+2) {$-$};
\draw [fill] (-3,-.433+2) circle [radius=.05];
\draw node [above left] at (-3,-.433+2) {$+$};

\draw[fill] (0.5,-1.567) circle [radius=0.05];
\draw node [above left] at (0.5,-1.567) {$-$};
\draw[fill] (-.5, -2.433) circle [radius=0.05];
\draw node [above left] at (-.5,-2.433) {$+$};
\draw [fill] (-2,-1.567) circle [radius=.05];
\draw node [above left] at (-2,-1.567) {$-$};
\draw [fill] (-3,-2.433) circle [radius=.05];
\draw node [above left] at (-3,-2.433) {$+$};

\begin{scope}[xshift=-.5cm,yshift=-.433cm]
\tqftdisc
\draw (1.618,.6266) -- (1.618,.6266 -2);
\draw (1.618-2.5,.6266+2) -- (1.618-2.5,.6266 );
\begin{scope}[yshift=-2cm, cm={1,0,.5,0.433,(0,0)}]
\draw (0,0) arc (270:400:1);
\draw [dashed] (.78,1.618) arc (400:450:1);
\end{scope}
\end{scope}
}
\newcommand{\tqftidev}{
\begin{scope}[cm={1,0,.5,0.433,(0,0)}]
\draw (0,0) arc (270:400:1);
\draw (.78,1.618) arc (400:450:1);
\end{scope}
\begin{scope}[yshift=-2cm, cm={1,0,.5,0.433,(0,0)}]
\draw (0,0) arc (270:400:1);
\draw [dashed] (.78,1.618) arc (400:450:1);
\end{scope}

\draw (1.618,.6266) -- (1.618,.6266 -2);

\begin{scope}[xshift=.5cm,yshift=.433cm]
\draw (-.5,-.433) -- (-.5, -2.433);
\draw (.5,.433) -- (0.5,-.2);
\draw [dashed] (.5,-.2) -- (0.5,-1.567);

\draw[fill] (0.5,.433) circle [radius=0.05];
\draw node [above left] at (0.5,0.433) {$+$};
\draw[fill] (-.5, -.433) circle [radius=0.05];
\draw node [above left] at (-.5,-.433) {$-$};
\draw[fill] (0.5,-1.567) circle [radius=0.05];
\draw node [above left] at (0.5,-1.567) {$+$};
\draw[fill] (-.5, -2.433) circle [radius=0.05];
\draw node [above left] at (-.5,-2.433) {$-$};
\end{scope}
}

\begin{document}
\title{Spin Hurwitz numbers and topological quantum field theory}

\author{Sam Gunningham}

\begin{abstract}

Spin Hurwitz numbers count ramified covers of a spin surface, weighted by the
size of their automorphism group (like ordinary Hurwitz numbers), but signed
$\pm 1$ according to the parity of the covering surface. These numbers were first defined by Eskin-Okounkov-Pandharipande in order to study the moduli of holomorphic differentials on a Riemann surface. They have also been related to Gromov-Witten invariants of of complex 2-folds by work of Lee-Parker and Maulik-Pandharipande. In this paper, we construct a (spin) TQFT which computes these numbers, and
deduce a formula for any genus in terms of the combinatorics of the Sergeev
algebra, generalizing the formula of Eskin-Okounkov-Pandharipande.
During the construction, we describe a procedure for averaging any TQFT over
finite covering spaces based on the finite path integrals of
Freed-Hopkins-Lurie-Teleman. 
\end{abstract}
\maketitle

Hurwitz numbers count ramified covers of surfaces weighted by the size of 
the automorphism group of the cover. Combinatorial formulas in terms of the representation 
theory of the symmetric group go back to Hurwitz \cite{Hur},  Frobenius \cite{Frob}, and 
Burnside \cite{Bur}.

Suppose $\Sigma$ is a closed surface equipped with a spin structure. Associated to $\Sigma$ is an element $p(\Sigma) \in \Z/2\Z$ known as the parity or Atiyah invariant. In this paper, we study a variant of Hurwitz numbers called \emph{spin Hurwitz numbers}, where ramified covers of $\Sigma$ are counted with a sign according to the parity of the total space of the cover (equipped with the pullback spin structure). One of the main results is the following formula for spin Hurwitz numbers as a sum over the set $\SP(n)$ of \emph{strict partitions} of $n$ (i.e. partitions into distinct parts).

\begin{theorem}\label{Theorem1}
The degree $n$ unramified\footnote{See Theorem \ref{Hurwitz} for a more general version involving ramified covers.} spin Hurwitz numbers of a closed spin surface $\Sigma$ are given by:
\[
2^{-n\chi(\Sigma)/2}\sum \limits_{\bnu \in \SP(n)} (-1)^{p(\Sigma) \ell(\bnu)}
\left( \frac{d(\bnu)}{n!}\right)^{\chi(\Sigma)},
\]
where
\[
d(\bnu) = 2^{n-\ell({\bnu})/2} \frac{n!}{\nu _1 ! \ldots \nu _\ell !} \prod _{p<q} \frac{\nu _q - \nu _p}{\nu _p + \nu _q},
\]
and $\ell(\bnu)$ denotes the length of the partition $\bnu$.
\end{theorem}

\begin{remark}
The numbers $d(\bnu)$ have a representation theoretic interpretation: namely they are the dimensions of the simple supermodules of the \emph{Sergeev algebra}, up to a factor of $\sqrt{2}$ if $\ell(\bnu)$ is odd (see Proposition \ref{propositionsergeev2}).
\end{remark}

Spin Hurwitz numbers were first introduced by Eskin-Okounkov-Pandharipande \cite{EOP}, who were interested in the volume of strata in the moduli space of differentials on a Riemann surface with prescribed singularities. Kontsevich-Zorich \cite{kontsevich_connected_2003} and Eskin-Masur \cite{eskin_asymptotic_2001} independently explained how these volumes are related to the asymptotics of the enumeration of connected branched coverings of a torus, where the ramification data of the covering matches the singularities of the differential. Eskin-Okounkov \cite{eskin_asymptotics_2001} were able to use this idea to compute the volumes in the case when the ramification data contained only partitions into even parts; the case when the ramification data consisted of partitions into odd parts required the introduction of spin Hurwitz numbers and was solved in \cite{EOP}. One of the main results of their paper was a group theoretic computation of the parity of a ramified covering of a spin surface. This result was applied in the case of a torus with odd spin structure to give a formula for the spin Hurwitz numbers in that case. In this paper, we will generalize the formula of \cite{EOP} to any genus (Theorem \ref{Hurwitz}).

Spin Hurwitz numbers have also appeared in the work of Lee-Parker \cite{LP1} and Maulik-Pandharipande \cite{MP} in Gromov-Witten theory. Suppose $X$ is a complex K\"ahler 2-fold with a smooth canonical divisor $\Sigma \subseteq X$. The normal bundle of $\Sigma$ in $X$ defines a spin structure on $\Sigma$. It was shown in \cite{LP1} that certain Gromov-Witten invariants of $X$ can be expressed in terms of spin Hurwitz numbers of $\Sigma$. More recently, Lee and Parker have proved a recursion formula that allows one to compute higher genus spin Hurwitz numbers in terms of the lower genus numbers using techniques from analysis and PDE \cite{LPP} \cite{lee_degree-three_2013}. The recursion formulas of \cite{LPP} also follow from the techniques of this paper (see Section \ref{leeparker}).

\begin{remark}
There does not appear to be an obvious connection between the appearance of spin Hurwitz numbers in the Gromov-Witten theory of 2-folds, and in the volume of the moduli spaces of differentials on a Riemann surface.
\end{remark}

The main tool of this paper is topological field theory (TQFT). In the 1990's, Dijkgraaf-Witten \cite{DW} and Freed-Quinn \cite{FQ} explained how to construct a topological quantum field theory from a finite gauge group. Hurwitz numbers are controlled by a 2-dimensional version of Dijkgraaf-Witten theory, with gauge group $S_n$). The formulas of Hurwitz, Frobenius, and Burnside can be recovered by analyzing the structure of this TQFT. To prove Theorem \ref{Theorem1} we will construct a (spin) TQFT taking values in the category of superalgebras:

\begin{theorem}\label{Theorem2}
For each positive integer $n$, there is a fully extended 2d spin TQFT $Z_n$ which assigns the following invariants:
\begin{itemize}
\item To a point, $Z_n$ assigns the Sergeev superalgebra, $\cY_n := \Cl_n \rtimes S_n$, where $\Cl_n$ denotes the $n$th complex Clifford algebra. 
\item To the anti-periodic spin circle, $Z_n$ assigns the even super vector space generated by \emph{odd} partitions of $n$ (i.e. partitions into odd parts). 
\item To the periodic spin circle, $Z_n$ a assigns the super vector space of \emph{strict} partitions (i.e. partitions into distinct parts) where the degree of a partition $\bmu$ is its length $\mod 2$. 
\item To a closed spin surface $\Sigma$, $Z_n$ assigns the spin Hurwitz numbers of $\Sigma$.
\end{itemize}
\end{theorem}

The spin TQFT $Z_n$ is a convenient way to encode all the information about spin Hurwitz numbers $\cH_n(\Sigma)$ for a spin surface $\Sigma$. The locality properties of TQFT mean that $Z_n$ is determined by the superalgebra $Z_n(pt) = \cY_n$, together with a linear map $\cZ(\cY_n) \to \C$. In fact, we can recover $Z_n$ from the following data: the set of simple supermodules $S=\{V^\bnu \mid \bnu \in \SP(n) \}$ of $\cY_n$, the decomposition $S=S_0 \sqcup S_1$ into type $M$ and type $Q$ supermodules (see Section \ref{Sergeev}), and complex number $t(\bnu)$ for each $\bnu \in \SP(n)$. 

Statements about spin Hurwitz numbers often reduce to formal properties of TQFT. For example, the spin Hurwitz formulas in Theorem \ref{Theorem1} are a special case of a more general formula for spin TQFT (Proposition \ref{inv}). Recursion formulas, such as those of Lee and Parker \cite{LPP}, follow directly from the existence of the TQFT $Z_n$ (see Section \ref{leeparker}). 

\subsection*{Outline of the paper}
Section \ref{introduction} we introduce some of the necessary preliminaries, and state the main results in full generality (Theorems \ref{Hurwitz} and \ref{TFT}).

In Section \ref{SpinTFT}, we give a general description of spin TQFTs and explain how to compute the invariants assigned to a closed surface by such a TQFT. The fundamental example of a spin TQFT is the parity theory which assigns $\pm 1$ to a closed spin surface according to its parity. The construction of the parity theory is left to Section \ref{AtiyahTheory}. In Section \ref{average}, we explain how to construct new TQFTs from old ones by averaging over $n$-fold covering spaces. The spin Hurwitz theory of Theorem \ref{TFT} is then constructed in Section \ref{finalsection} by averaging the parity theory over $n$-fold covers. Theorem \ref{Hurwitz} is proved by applying the computation from Section \ref{SpinTFT} to the spin Hurwitz TQFT.

\subsection*{Acknowlegements}
I would first like to thank C. Elliott, D. Freed, E. Getzler, O. Gwilliam, J. Lurie, T. Matsuoka, A. Okounkov, M. Shapiro, H.L. Tanaka, and J. Wolfson for helpful conversations and comments during the preparation of this work. I would also like to thank S. Stolz and P. Teichner for their interest and comments, and T. Parker for sharing his ideas and work in progress with J. Lee. This paper has greatly benefited from the detailed and helpful comments of (an) anonymous referee(s).

I am especially grateful to my advisor, D. Nadler who first taught me how to count covers using TQFT, for his consistent guidance and insight. Finally, I am indebted to D. Treumann, from whom the initial idea for this project came from. I have greatly benefited from his ideas and suggestions during the many conversations we have had on this subject, and from his wisdom and encouragement. 

\section{Preliminaries and Main Results}\label{introduction}

\subsection{Ordinary Hurwitz numbers}\label{SSOrdinaryHurwitz}
It will be helpful to recall how ordinary Hurwitz numbers are defined and computed using TQFT. The results in sections \ref{SSOrdinaryHurwitz} and \ref{SSTQFT} are reasonably well known, but we present them here in order make analogies with our own results.

Let $\PP (n)$ denote the set of partitions of $n$, i.e. $\bmu = (\mu _1 \leq \ldots \leq \mu _\ell)$,
where $\mu _1 + \ldots + \mu _\ell = n$. We write $\ell(\bmu) = \ell$ for the \emph{length}
of $\bmu$. Note that isomorphism classes of $n$-fold covering spaces of a circle
are in bijection with $\PP(n)$; we call the partition corresponding to such a
cover, the \emph{ramification datum} of the cover. 

Fix a closed, oriented surface $\Sigma$ with marked points $p_1 , \ldots , p_k$.
\begin{definition}
 An \emph{$n$-fold ramified cover} of $\Sigma$, with ramification data $\bmu ^1  \ldots , \bmu ^k \in \PP(n)$ 
at $p_1 , \ldots , p_k$ is a surface $\widetilde{\Sigma}$ with a
finite, continuous map $\widetilde{\Sigma} \to \Sigma$ which is an n-fold covering space
over $\Sigma - \{ p_1 , \ldots , p_k \}$, and such that the restriction to a small circle around $p_i$ 
has ramification datum $\bmu ^i$. Define the \emph{Hurwitz numbers} by
\[
 \widehat{\cH} _n (\Sigma ; \bmu ^1 , \ldots , \bmu ^k) = \sum \frac{1}{\# \Aut
(\widetilde{\Sigma}/\Sigma)},
\]
where the sum is taken over isomorphism classes of ramified covers $\widetilde{\Sigma} 
\to \Sigma$ with the specified ramification data. We will write $\widehat{\cH} _n(\Sigma, k)$
for the associated linear functional on $\C [\PP(n)]^{\otimes k}$
\end{definition}

Recall that the set $\PP(n)$ naturally parameterizes both conjugacy classes and irreducible representations of $S_n$.
\begin{definition}
 Let $\widehat{V}^\bnu$ denote the irreducible representation of $S_n$ corresponding to $\bnu \in \PP(n)$. The \emph{central characters} $\widehat{f}^\bnu _\bmu$ of $S_n$ are defined to be
 the constant multiple of the identity by which the conjugacy class $\bmu$ acts on the 
 representation $\widehat{V}^\bnu$.
\end{definition}

The following formulas go back to Hurwitz, Frobenius, and Burnside (in various forms).

\begin{theorem}[\cite{Hur} \cite{Frob} \cite{Bur}] \label{OrdinaryHurwitz}
 The numbers $\widehat{\cH} _n (\Sigma ; \bmu ^1 , \ldots , \bmu ^k)$ are given by the following
 formula
 \[
\sum \limits_{\bnu \in \PP(n)} \left( \prod _{i=1} ^k \hat{f} _{\bmu ^i} ^\bnu \right) \left( \frac{\dim \widehat{V}^\bnu}{n!} \right)^{\chi(\Sigma)}.
 \]
\end{theorem}
Moreover, we have the following expressions for the dimensions and central characters of the symmetric
group (see e.g. \cite{Ful}):
\[
 \dim \widehat{V}^\bnu = \frac{n!}{L_1 ! \ldots L_\ell !} \prod _{i<j} (L_i - L_j),
\]
where $L_i = \mu _i + \ell - i$. The $\hat{f}^\bnu _\bmu$ satisfy
\[
 S_\bnu (X) = \sum _{\bmu \in \PP(n)} \frac{\dim \widehat{V} ^\bnu}{n!}\hat{f}^\bnu _\bmu p_\bmu (X)
\]
where $S_\bnu (X)$ are the Schur functions, $p_m(X) = \sum _{j\ge 1} x^m _j$ is a 
power sum function, and $p_{\bmu} (X) = p_{\mu _1}(X) \ldots p_{\mu _{\ell}}(X)$.

\subsection{Topological quantum field theory}\label{SSTQFT}
Let $\Sigma$ denote a closed, oriented surface and suppose that $\Sigma$ can be decomposed as $\Sigma _1 \cup _N \Sigma _2$ where the $\Sigma _{i}$ are surfaces glued along a common boundary $N$. The key observation required to prove Theorem \ref{OrdinaryHurwitz} is that the Hurwitz numbers of $\Sigma$ can be computed in terms of the Hurwitz numbers of the $\Sigma _i$ (with extra marked points corresponding to the boundary components $N$). One convenient way to organize this relationship is via a \emph{topological quantum field theory} (TQFT).

Let $\Cob^{Or}$ denote the category whose objects are closed, oriented 1-manifolds, and a morphism $N_1 \to N_2$ is an oriented surface $\Sigma$ with an identification $\partial \Sigma \equiv N_1 \sqcup \overline{N_2}$ (up to diffeomorphism). This has a symmetric monoidal structure given by disjoint union of manifolds. Let $\Vect$ denote the category of complex vector spaces with its symmetric monoidal structure given by tensor product.

\begin{definition}
A TQFT is a symmetric monoidal functor $Z: \Cob^{Or} \to \Vect$.
\end{definition}

Given any TQFT $Z$, the vector space $V=Z(S^1)$ naturally has the structure of a commutative algebra, where the multiplication is given by applying the functor $Z$ to the cobordism $S^1 \sqcup S^1 \to S^1$ given by a pair of pants (a sphere with three discs removed). Moreover, there is a map $t : V \to \C$ given by applying $Z$ to the disc, considered as a cobordism $S^1 \to \emptyset^1$. The map $t$ makes $V$ into a \emph{commutative Frobenius algebra}, i.e. the composite $V \otimes V \xrightarrow{mult} V \xrightarrow{t} \C$ is a non-degenerate inner product on $V$.

If this algebra is \emph{semisimple}, by Wedderburn's theorem $V$ has a basis $e_s$ of orthogonal idempotents. 
\begin{proposition}\label{OrdinaryCalc}
Let $Z$ be a TQFT with the algebra $V=Z(S^1)$ semisimple, and let $e_s$, $s\in S$, be the basis of orthogonal idempotents of $V$. Then for any any closed, oriented surface $\Sigma$, we have:
\[
Z(\Sigma) = \sum _{s\in S} t (e_s)^{\chi (\Sigma)/2}.
\]
\end{proposition}
Proposition \ref{OrdinaryCalc} is proved by decomposing $\Sigma$ into discs and pairs of pants (see Proposition \ref{inv} for more details). The following theorem gives a precise description of how Hurwitz numbers behave under such decompositions of $\Sigma$.

\begin{theorem}\label{OrdinaryTFT}
For each positive integer $n$, there is a TQFT $\widehat{Z}_n$ so that:
\begin{itemize}
\item To a circle $\widehat{Z}_n$ assigns the vector space $\C [\PP(n)]$.
\item To a surface $\Sigma$ with $k$ marked points, considered as a cobordism $(S^1) ^{\sqcup k} \to \emptyset^1$, $\widehat{Z}_n$ assigns $\widehat{\cH}_n(\Sigma):\C[\PP(n)] ^{\otimes k} \to \C$.
\end{itemize}
Moreover, the commutative Frobenius algebra structure on $\C[\PP(n)]$ is given by identifying it with the algebra of class functions $\C[S_n]^{S_n}$ such that a partition $\bmu$ maps to the conjugacy class of $S_n$ with cycle type $\bmu$. The map $t$ takes a class function $f$ to $(1/n!)f(1)$.
\end{theorem}

The algebra $\C[S_n]^{S_n}$ is semisimple; the change of basis matrix between the conjugacy classes $\bmu$ and the orthogonal idempotents $e_\bnu$ is given by the central characters $\hat{f}^\bnu _\bmu$, and $t (e_\bnu) = \left(\frac{\dim \widehat{V}^\bnu}{n!} \right)^2$. Combining Theorem \ref{OrdinaryTFT} with Proposition \ref{OrdinaryCalc}, we deduce the formula for Hurwitz numbers in Theorem \ref{OrdinaryHurwitz}.

\begin{remark}
There is a natural analogue of this story, where $n$-fold covers are replaced by $G$-bundles for any finite group $G$, see \cite{FHLT}.
\end{remark}

\subsection{Spin Hurwitz numbers}\label{SSSHN}
Now suppose $\Sigma$ is a closed surface equipped with a \emph{spin structure}. Such a structure is equivalent to choosing either of the following pieces of data:
\begin{itemize}
 \item A quadratic form on $H^1(\Sigma , \ZZ/2\ZZ)$ which refines the intersection pairing,
 \item A square root, $K^{1/2}$ of the canonical bundle (for any complex structure on $\Sigma$).
\end{itemize}

\begin{definition}\label{definitionparity}
 In terms of the descriptions above, the \emph{parity} of $\Sigma$ (with its spin structure) is defined to be
 \begin{itemize}
  \item The Arf invariant of the corresponding quadratic form on $H^1(\Sigma, \ZZ/2\ZZ)$ ,
  \item The dimension mod 2 of $H_{hol}^0(\Sigma , K^{1/2})$ .
 \end{itemize}
\end{definition}
(See \cite{Jo} \cite{At} \cite{Mum}).

Recall that there are two spin structures on the circle $S^1$: the \emph{anti-periodic} (or \emph{Neveu-Schwarz}) circle $S^1 _{ap}$ which bounds a spin disc, and the \emph{periodic} (or \emph{Ramond}) circle $S^1 _{per}$ which does not \cite{LM}.

Two subsets of the set $\PP(n)$ of partitions of $n$ play an important role in this paper:
\begin{itemize}
\item The set of \emph{odd partitions}, $\OP (n) = \left\{ \mu \in \PP (n) : \ \mu_i \ \text{are all odd} \right\}$.
\item The set of \emph{strict partitions}, $\SP(n) = \left\{ \mu \in \PP(n) : \ \mu_i \ \text{are all distinct} \right\}$.
\end{itemize}
 We will see that these are analogues of the dual roles played by $\PP(n)$ for ordinary Hurwitz numbers: indexing conjugacy classes and representations of the symmetric group.\footnote{It was shown by Euler that the cardinalities of $\SP(n)$ and $\OP(n)$ are equal.}

Suppose we have a cover $\widetilde{\Sigma} \to \Sigma$ ramified at $p_1 , \ldots , p_k$. We can canonically lift the spin structure on $\Sigma$ to $\widetilde{\Sigma}$ precisely when the ramification data are all odd partitions (in that case, every component of the covering space over each boundary circle is anti-periodic, thus the spin structure extends over the disc it bounds).

\begin{definition}\label{Hurwitznumbers}
The \emph{spin Hurwitz numbers} are defined by:
\[
\cH_n(\Sigma ; \bmu ^1 , \ldots , \bmu ^k) = \sum \frac{(-1)^{p(\widetilde{\Sigma})}}{ \# \Aut{\left( \widetilde{\Sigma}/\Sigma \right)}} ,
\]
where the sum is taken over isomorphism classes of branched covers 
$\widetilde{\Sigma}$ of $\Sigma$ with ramification data $\bmu^1 , \ldots , \bmu ^k \in \OP(n)$ 
at $p_1 , \ldots p_k$. We will write $\cH_n(\Sigma , k)$ for the associated functional 
on the vector space $\C [\OP (n)]^{\otimes k}$.
\end{definition}

One of the main results of this paper is a combinatorial expression for the spin Hurwitz numbers. The unramified case was Theorem \ref{Theorem1}; this involved the numbers $d(\bnu)$, for $\bnu \in \SP(n)$. To state the general case we will need to introduce the matrix $f^\bnu _\bmu$, for $\bmu \in \OP(n)$ and $\bnu \in \SP(n)$. This matrix has a combinatorial definition in terms of the Schur $Q$-functions, a certain family of symmetric polynomials $Q_\bnu(X) = Q_\bnu(x_1,x_2, \ldots)$ which were introduced by Schur \cite{Sch} in his study of the spin representations of the symmetric group (the representation theoretic meaning of $f^\bnu _\bmu$ will be explained in Section \ref{centralcharacters}). 

The $Q$-functions were defined by Schur as follows (see also \cite{Joz2}). First consider
\[
Q(t) = \sum_{n\geq 0} Q_n(X)t^n := \prod _{j\geq 0} \frac{1+x_jt}{1-x_jt}. 
\]
Now we define $Q_{pq}(X)$ by
\[
Q(r, s) = \sum _{p,q \geq 0} Q_{pq}(X) r^p s^q := (Q(r) Q(s) -1)\frac{r-s}{r+s}.
\]
Finally, for $\bnu \in \SP(n)$, we set $Q_\bnu(X) = Pf(Q_{\nu_i \nu_j}(X))$ (if $\ell(\bnu)$ is odd, we set $\nu_{\ell(\bnu)+1} = 0$ so that $(Q_{\nu_i \nu_j}(X))_{ij}$ is an even order antisymmetric matrix and its Pfaffian is well defined).

\begin{theorem}[The spin Hurwitz formula]\label{Hurwitz}
The spin Hurwitz numbers $\cH_n(\Sigma, \bmu ^1 , \ldots , \bmu ^k)$ are given by:
\[
2^{(\sum_i(\ell(\bmu^i) -n) -n\chi(\Sigma))/2} \sum_{\bnu \in \SP(n)} (-1)^{p(\Sigma) \ell(\bnu)} 
\left(\prod_i f^\bnu _{\bmu^i} \right)\left(\frac{d(\bnu)}{n!}\right)^{\chi(\Sigma)},
\]
where:
\[
d(\bnu) = 2^{n -\ell({\bnu})/2} \frac{n!}{\nu _1 ! \ldots \nu _\ell !} \prod _{p<q} \frac{\nu _q - \nu _p}{\nu _p + \nu _q},
\]
and the numbers $f_{\bmu} ^{\bnu}$ satisfy
\[
Q_{\bnu} (X) = 2^{\ell (\bnu)/2} \sum _{\bmu \in \OP (n)} \frac{d(\bnu)}{n!} f_{\bmu} ^{\bnu} p_{\bmu} (X),
\]
($p_m(X) = \sum_{j\geq 1} x_j ^m$ is a power sum function, and $p_\bmu(X) = p_{\mu_1}(X) \ldots p_{\mu_\ell}(X)$).
\end{theorem}

\subsection{Semisimple superalgebras}\label{Sergeev}
In this section we will recall the basic theory of semisimple superalgebras, after J\'ozefiak \cite{Joz0}.

We will use the term \emph{super vector space} to denote a $\Z/2\Z$ graded complex vector space. Let $S\Vect$ denote the category of super vector spaces and (degree preserving) linear maps. This has a natural symmetric monoidal structure, where the symmetric structure incorporates the sign rule: the isomorphism $V \otimes W \to W \otimes V$ is given by $v\otimes w \mapsto (-1)^{|v||w|} w \otimes v$ where $v$ and $w$ are homogeneous elements of degree $|v|$ and $|w|$. 

A \emph{superalgebra} is an algebra object in $S\Vect$. The notions of supercommutative, opposite superalgebra $A^{op}$, supercentre $\cZ(A)$ and superabelianization $\Ab(A)$ of a superalgebra $A$ in this paper are all defined following the sign rule. A supermodule for a superalgebra is just a module with a compatible $\Z/2\Z$ grading. Two superalgebras $A$ and $B$ are Morita equivalent if there are superbimodules $\lrsub AMB$ and $\lrsub BNA$ such that $M \otimes_B N \simeq A$ and $N \otimes _A M \simeq B$. Given a superalgebra $A$, and two supermodules $M$ and $N$, we denote by $\Hom_A(M,N)$ the space of degree preserving $A$-linear maps, and by $\HOM_A(M,N)$ the space of \emph{all} $A$-linear maps with its natural $\Z/2\Z$-grading.

A superalgebra is called \emph{semisimple} if every supermodule is a direct sum of simple supermodules. Wedderburn theory for semisimple superalgebras states that every simple superalgebra is isomorphic to one of the following:
\begin{itemize}
\item $M(r,s) = \END ( \C ^r \oplus \C^s[1])$, the superalgebra of (non degree preserving) endomorphisms.
\item $Q(d)$, the subalgebra of $M(d,d))$, consisting of matrices of block form $\left( \begin{array}{cc} C&D \\ D&C \end{array}\right)$.
\end{itemize}
Every semisimple superalgebra is isomorphic to a product of such simple superalgebras, with one simple factor for each simple module. 

The superalgebras $M(r,s)$ are all Morita equivalent to the trivial algebra $\C=M(1,0)$, and the superalgebra $Q(d)$ are Morita equivalent to $Q(1) = \C [\eta]$ where $\eta$ has degree 1 and $\eta ^2 =1$. 
Simple supermodules for a semisimple algebra will be referred to as type $M$ or type $Q$ depending on which factor they correspond to.

\begin{remark}
The supercentre $\cZ(A)$ of every semisimple superalgebra $A$ is in purely even degree, and this has a basis of orthogonal idempotents corresponding to simple supermodules. The superabelianization $\Ab(A)$ (the quotient of $A$ by the subspace of all supercommutators) is a vector space of the same dimension, however the summands corresponding to supermodules of type $M$ are in even degree and the those corresponding to supermodules of type $Q$ are in odd degree (to verify this claim, it is enough to check for the superalgebras $\C$ and $Q(1)$).
\end{remark}

The $n$th Clifford algebra $\Cl_n$ is defined as the tensor product $Q(1)^{\otimes n}$. Explicitly, $\Cl_n$ is generated by odd elements $\eta_1 , \ldots, \eta_n$ where $\eta_i ^2 = 1$, and $\eta_i \eta_j = - \eta_j \eta_i$ whenever $i \neq j$. The Clifford algebra $\Cl_n$ is a simple superalgebra. When $n$ is even, $\Cl_n$ is Morita equivalent to $\C$. When $n$ is odd, $\Cl_n$ is Morita equivalent to $\Cl_1$.

\begin{definition}\label{SergeevAlg}
The symmetric group $S_n$ acts on $\Cl_n$ by permuting the generators $\eta_i$. The \emph{Sergeev superalgebra} $\cY_n$ is defined to be the semidirect product $\Cl_n \rtimes S_n$. 
\end{definition}

\subsection{Spin representations and twisted group superalgebras}\label{subsectionspinreps}

Suppose $G$ is a finite group, together with a central extension
\[
 \langle \varepsilon \rangle \hookrightarrow \widetilde{G} \xrightarrow{\pi} G.
\]
where $\varepsilon ^2 = 1 $. A \emph{spin representation of $G$} is defined to be a representation of $\widetilde{G}$ in which $\varepsilon$ acts by $-1$. Spin representations of $G$ are the same thing as modules for the \emph{twisted group algebra}
\[
\cT(G) = \C[\widetilde{G}]/(\varepsilon + 1).
\]
(we suppress the data of the central extension from the notation).

Suppose in addition that $G$ is equipped with a $\Z/2\Z$-grading determined by an index 2 subgroup $G_0$ of $G$. Then $\widetilde{G}$ acquires a grading via $\widetilde{G}_0$ is the preimage of $G_0$, and the twisted group algebra is endowed with the structure of a superalgebra. For each supermodule $V$ of $\cT(G)$, let $|V|$ denote the spin representation of $G$ obtained by forgetting the grading.

\begin{proposition}[\cite{Joz0}]
The superalgebra $\cT(G)$ is a semisimple superalgebra. For each simple supermodule $V$ of $\cT(G)$ ,
\begin{itemize}
\item if $V$ is of type $M$, $|V|$ is also simple, and
\item if $V$ is of type $Q$, $|V|$ splits as a direct sum of two non-isomorphic simple modules $|V|_+$ and $|V|_-$.
\end{itemize}
Every simple spin representation of $G$ arises as $|V|$, $|V|_+$, or $|V|_-$ for a unique simple supermodule $V$.
\end{proposition}

Thus the representation theory of the superalgebra $\cT(G)$ encodes the spin representations of $G$. Here are the main examples that we will consider in this paper:
\begin{enumerate}
\item Let $R_n$ denote the group $(\Z/2\Z)^n$, and label the generators $\eta_1 , \ldots, \eta_n$. This group has a grading in which each $\eta_i$ has degree 1, and a central extension $\widetilde{R_n}$ where the central element $\varepsilon$ satisfies $\eta_i \eta_j = \varepsilon \eta_j \eta_i$ when $i \neq j$. The twisted group superalgebra $\cT(R_n)$ is isomorphic to $\Cl_n$.
\item Let $B_n$ denote the \emph{hyperoctahedral group} $R_n \rtimes S_n$ with a grading induced from the grading on $R_n$. The group $\widetilde{B_n} = \widetilde{R_n} \rtimes S_n$ defines a central extension of $B_n$ called the \emph{Sergeev group}. The twisted group algebra $\cT(B_n)$ is canonically identified with the Sergeev algebra $\cY_n = \Cl_n \rtimes S_n$.
\item The symmetric group $S_n$ also has a central extension $\widetilde{S_n}$. To define $\widetilde{S_n}$, let $t_1 , \ldots, t_{n-1}$ denote the standard Coxeter generators of $S_n$. Then $\widetilde{S_n}$ is generated by the $t_i$ together with the central involution $\varepsilon$, subject to the usual relations $t_i^2 =1$ and $t_i t_{i+1} t_i = t_{i+1} t_i t_{i+1}$, but now we have $t_i t_j = \varepsilon t_j t_i$ when $|i-j| >1$. The symmetric group has a grading given by the parity. 
\end{enumerate}

\subsection{Central characters of spin representations}\label{centralcharacters} 
Recall that for a finite group $H$, each conjugacy class $C$ determines a central element in the group algebra. This acts on a simple representation $M$ by a scalar multiple $f^M_C$ of the identity (equivalently, $f^M_C$ is the change of basis matrix between the basis of conjugacy classes and the basis of orthogonal idempotents in the centre of the group algebra of $H$). For a given simple representation $M$, the number $f^M_C$ is called the \emph{central character} of $M$; it is related to the \emph{character} (trace) $\chi_M$ by the equation 
\[
\chi_M (C) = \frac{\dim(M)}{\# C} f^M_C.
\]

Now let us define the central characters of \emph{spin} representations of a group $G$ (with a grading and central extension as in Section \ref{subsectionspinreps}). Let $C$ be a conjugacy class in $G$, and consider the preimage $\pi^{-1}(C)$ of $C$ in $\widetilde{G}$. There are two possibilities: either $\pi^{-1}(C)$ is a conjugacy class in $\widetilde{G}$, or $\pi^{-1}(C)$ splits as a disjoint union of two conjugacy classes. We will say that the class $C$ \emph{splits in $\widetilde{G}$} if the second case occurs.

If $C$ splits in $\widetilde{G}$, choose a conjugacy class $D$ in $\widetilde{G}$ such that $\pi^{-1}(C)= D \sqcup \varepsilon D$. Then the image of $D$ in $\cT(G)$ is in the supercentre of $\cT(G)$, and the collection of such elements (as $C$ varies over the conjugacy classes in $G$ of the second type) forms a basis for $\cZ(\cT(G))$. If $C$ does not split in $\widetilde{G}$, then the image of $\pi^{-1}(C)$ in $\cT(G)$ is zero, as $\varepsilon \pi^{-1}(C) = \pi^{-1}(C)$. Such conjugacy classes do not contribute to the character, and we will disregard them.

\begin{remark}
This basis is canonical only up to sign as we could have chosen the image of the class $\varepsilon D$ instead of $D$. However, in the cases of interest, there will always be a preferred choice of lift $D$.
\end{remark}

\begin{definition}
Given a simple supermodule $V$ of $\cT(G)$ its central character $f^V_C$ is the scalar multiple of the idenitity on which the image of the (chosen) lift $D$ of $C$ acts on $V$.
\end{definition}

\begin{proposition}[\cite{Ser} \cite{Joz2}] \label{propositionsergeev}
There is a bijection between the set of strict partitions $\lambda \in \SP(n)$ and supermodules $V^\lambda$ of $\cY_n$ such that $V_\lambda$ is of type $M$ if $\ell(\lambda)$ is odd, and of type $Q$ is $\ell(\lambda)$ is even.
\end{proposition}

Recall that $\cY_n$ is canonically identified with the twisted group algebra $\cT(B_n)$ of the hyperoctahedral group. Conjugacy classes $C_\bmu$ in the hyperoctahedral group which split in the Sergeev group $\widetilde{B_n}$ are indexed by odd partitions $\bmu \in \OP(n)$. They are precisely the conjugacy classes generated by a permutation of cycle type $\bmu$ under the embedding $S_n \hookrightarrow B_n$. We also have an embedding $S_n \hookrightarrow \widetilde{B_n}$ and the conjugacy class $D_\bmu$ of a permuation of cycle type $\bmu$ in $\widetilde{B_n}$ defines a lift of $C_\bmu$. 

Thus we can define the central characters $f^{V^\bnu} _{C_{\bmu}}$ of the Sergeev algebra $\cY_n$. Recall that we defined numbers $f^\bmu_\bnu$ in Section \ref{SSSHN} in terms of the Schur $Q$-function.

\begin{proposition}[\cite{Ser} \cite{Joz2}]\label{propositionsergeev2}
We have $f^{V^\bnu}_{C_{\bmu}} = f^\bnu _\bmu$. Moreover, the numbers $d(\bnu)$ are equal to $\dim V^\bnu$ if $\ell(\bnu)$ is even, and $(1/\sqrt{2}) \dim V^\bnu$ if $\ell(\bnu)$ is odd.
\end{proposition}

Let $\Delta_\bmu$ denote the image of the conjugacy class $D_\bmu$ in $\cZ(\cY_n)$, where $\bmu \in \OP(n)$. It will be useful to have an explicit description of the elements $\Delta_\bmu$. Given an $m$-cycle $\sigma = (a_1 \ldots a_m)$ in $S_n$ where $m$ is odd, define the element $\Delta_\sigma \in \cY_n$ by
\[
\sum (\eta_{a_1}^{c_1} \ldots \eta_{a_m} ^{c_m} ) \otimes \sigma,
\]
where the sum is over $m$-tuples $(c_1 , \ldots , c_m) \in (\Z/2\Z)^m$ such that $c_1 + \ldots c_m = 0 \in Z/2\Z$.
Now, given any element $\sigma \in S_n$ whose cycle type is an odd partition of $n$, define $\Delta_\sigma$ by $\Delta_{\sigma_1} \ldots \Delta_{\sigma_\ell}$ where $\sigma_1 \ldots \sigma_\ell$ is a factorization of $\sigma$ into disjoint cycles. 
\begin{proposition}\label{propositiondelta}
The element $\Delta_\bmu \in \cZ(\cY_n)$ is given by $\sum \Delta_\sigma$, where the sum is over $\sigma \in S_n$ with cycle type $\bmu$. 
\end{proposition}

For example $\Delta_{(3)} \in \cZ(\cY_3)$ is given by
\begin{align*}
1 \otimes (123) + \eta_1 \eta_2 \otimes (123) + \eta_2 \eta_3 \otimes (123) + \eta_3 \eta_1 \otimes (123)\\
1 \otimes (132) + \eta_1 \eta_3 \otimes (132) + \eta_3 \eta_2 \otimes (132) + \eta_2 \eta_1 \otimes (132).
\end{align*}

\begin{remark}
Proposition \ref{propositionsergeev2} above gives a combinatorial description of the spin characters of the hyperoctahedral group in terms of Schur $Q$-functions. However, Schur's original use of the $Q$-functions was to describe the spin characters of the symmetric group. The relationship between the spin characters of $S_n$ and $B_n$ is characterized by an isomorphism of superalgebras (see \cite{Yam} \cite{Kle}):
\[
\cY_n := \Cl_n \rtimes S_n \cong \Cl_n \otimes \cT(S_n).
\]
\end{remark}



\subsection{Extended TQFT and $2$-categories}\label{extended}
The definition of TFQT given in Section \ref{SSTQFT} assigns a vector space to a closed $1$-manifold and a linear map to a $2$-dimensional cobordism. A closed $2$-manifold is thus assigned a number, thought of as a linear map $\C \to \C$. There is a natural way to extend this definition which assigns an algebra (or alternatively, a linear category) to a $0$-manifold, and a bimodule to a $1$-dimensional cobordism. The vector space assigned to a closed $1$-manifold in this setting can be thought of as a $\C-\C$-bimodule. To properly formulate the definition of extended TQFT, we will need to work with $2$-categories.

Throughout this paper, the term \emph{2-category} will mean a weak 2-category, (or bicategory). We refer the reader to \cite{Ben} \cite{SP} for a more detailed discussion of 2-categories, but 
briefly a 2-category $\cC$ consists of the following data:
\begin{itemize}
\item A set, $\cC_0$ of \emph{objects};
\item A $1$-category $\Hom _\cC (x,y)$ of \emph{$1$-morphisms} for each $x,y \in \cC_0$ (the morphisms in these $1$-categories are called \emph{$2$-morphisms} and composition is referred to as \emph{vertical composition});
\item A functor $\circ: \Hom_\cC(x,y) \times \Hom_\cC(y,z) \to \Hom_\cC(x,z)$ called \emph{horizontal composition};
\item Identity $1$-morphisms $1_x \in \Hom(x,x)$ for each object $x$;
\item Natural isomorphisms called \emph{left and right unitors} for each pair of objects, establishing the expected properties of the identity $1$-morphism, and a natural transformation called an \emph{associator}, establishing the associativity property of horizontal composition. These satisfy some natural coherence axioms.
\end{itemize}
A symmetric monoidal structure on a 2-category $\cC$ is a functor $\otimes : \cC \times \cC \to \cC$ with fixed isomorphisms $x \otimes y \cong y \otimes x$, together with a unit object $1_\cC$, and various other natural transformations relating these objects (see \cite{SP}, Section 2.2). There is a notion of a \emph{symmetric monoidal functor} between two symmetric monoidal 2-categories.

Our primary examples of symmetric monoidal $2$-categories are $2$- categories of cobordisms. For example, the $2$-category $\Bord ^{or}$ of oriented cobordisms can be described as follows (see \cite{SP} for a more careful definition).
\begin{itemize}
\item Objects of $\Bord^{Or}$ are oriented 0-manifolds, i.e. disjoint unions of points equipped with an orientation. 
\item A 1-morphism between objects $N_0$ and $N_1$ is a cobordism between $N_0$ and $N_1$, i.e. an oriented 1-manifold with boundary, $M$ together with an oriented diffeomorphism $\partial M \cong \overline{N_0} \sqcup N_1$, where $\overline{N}$ denotes the opposite orientation on $N$. 
\item If $M_0$ and $M_1$ are both 1-cobordisms between $N_0$ and $N_1$, then a 2-morphism between $M_0$ and $M_1$ is a cobordism $\Sigma$ between $M_0$ and $M_1$ which is trivial on the boundary; more precisely, $\Sigma$ is an oriented $2$-manifold with corners together with an isomorphism 
\[
\partial \Sigma \simeq \overline{M_0} \sqcup_{N_0 \sqcup \overline{N_1}} \left( (\overline{N_0} \sqcup N_1) \times [0,1] \right)  \sqcup_{\overline{N_0} \sqcup N_1} M_1.
\]
\item The horizontal and vertical compositions are both described by gluing cobordisms along the relevant parts of the boundary (this requires some care to make precise).
\item The symmetric monoidal structure is given by disjoint union of manifolds.
\end{itemize}

Note that the 1-category $\Hom_{\Bord^{Or}}(\emptyset ^0, \emptyset^0)$ is equivalent to $\Cob^{Or}$. We would to extend the notion of TQFT from Section \ref{SSTQFT} to a symmetric monoidal functor from $\Bord^{Or}_2$ to a suitable target symmetric monoidal 2-category $\cC$, with the property that $\Hom_{\cC}(1_\cC, 1_\cC)$ is equivalent to $\Vect$. A common choice for $\cC$ is the $2$-category $\Alg$ which can be described as follows.
\begin{itemize}
\item Objects of $\Alg$ are $\C$-algebras,
\item A 1-morphism between objects $A$ and $B$ is an $A-B$-bimodule, $M$.
\item If $M$ and $M^\prime$ are both $A-B$-bimodules, then a 2-morphism between $M$ and $M^\prime$ is a homomorphism of bimodules.
\item Horizontal composition is given by relative tensor product, and vertical composition is given by composition of homomorphisms. 
\item The symmetric monoidal structure is given by tensor product of algebras.
\end{itemize}

\begin{definition}
A \emph{fully extended, oriented, 2d TQFT} is a symmetric monoidal functor $\Bord^{Or} \to \Alg$.
\end{definition}

\begin{example}
The Hurwitz TQFT from Theorem \ref{OrdinaryTFT}, extends to a fully extended TQFT which assigns the group algebra $\C[S_n]$ to a point. 
\end{example}



\subsection{Spin TQFT} 
Just as ordinary Hurwitz numbers are controlled by a TQFT, spin Hutrwitz numbers are controlled by a \emph{spin} TQFT. The $2$-category $\Bord^{Spin}$ is defined similarly to $\Bord^{Or}$, but all manifolds are now endowed with a spin structure (see \cite{SP}, Section 3.4 for more details). To allow for interesting spin TQFTs, we will also change the target $2$-category from $\Alg$ to $S\Alg$, whose objects the algebras and bimodules come with a $\Z/2\Z$-grading, and morphisms of bimodules are compatible with that grading.

\begin{definition}
A \emph{2d spin Topological Quantum Field Theory} (or TQFT) is a symmetric monoidal functor $\Bord ^{spin} \to S\Alg$.
\end{definition}


The following theorem is analogous to Theorem \ref{OrdinaryTFT} for ordinary Hurwitz numbers.

\begin{theorem}[The spin Hurwitz theory]\label{TFT}
For each positive integer $n$, there is a fully extended 2d spin TQFT $Z_n$ which assigns the following invariants:
\begin{itemize}
\item To a point, $Z_n$ assigns the Sergeev algebra, $\cY_n$. 
\item To the anti-periodic spin circle, $Z_n$ assigns the even super vector space $\C[\OP (n)]$. This can be identified with the supercentre $\cZ(\cY_n)$. 
\item To the periodic spin circle, $Z_n$ a assigns the super vector space $\C[\SP (n)]$ where the degree of $\bmu \in \SP(n)$ is $\ell (\bmu) \mod 2$. This can be identified with the superabelianization $\Ab(\cY _n)$.
\item To a closed spin surface $\Sigma$ with $k$ punctures, considered as a cobordism $\left( S^1 _{ap} \right) ^{\sqcup k} \to \emptyset^1$, $Z_n$ assigns $\cH_n(\Sigma , k): \C [\OP(n)]^{\otimes k} \to \C$.
\end{itemize}
\end{theorem}
\begin{remark}
Our construction gives a canonical basis for $Z_n(S^1 _{ap})$ indexed by $\OP(n)$, which is necessary to identify the value of $Z_n$ on a punctured surface with spin Hurwitz numbers. Strictly speaking, this basis is not an invariant of the functor $Z_n$ (it is not a Morita invariant of $\cY_n$), rather it is a feature of our construction of it. It is also important to note that the basis of $Z_n(S^1 _{per})$ given above is actually only well defined up to sign.
\end{remark}

\section{2d spin TQFTs}\label{SpinTFT}

\subsection{Dualizable and full dualizable objects}\label{dualizable}
Here we recall the notion of dualizability and full dualizability (see \cite{Lur} and Section 2 of \cite{BZN} for more details). 
\begin{definition}
Let $(\cC , \otimes)$ be a symmetric monoidal 2-category. An object $x$ in $\cC$ is \emph{dualizable} if there is another object $x^\vee$ and morphisms $ev: x \otimes x^\vee \to 1_\cC$ and $coev: 1_\cC \to x^\vee \otimes x$, such that:
\[
x \xrightarrow{1_x \otimes coev} x\otimes x^\vee \otimes x \xrightarrow{ev \otimes 1_x} x
\]
is isomorphic to $1_x$ , and
\[
x^\vee \xrightarrow{ coev \otimes 1_{x^\vee}} x^\vee \otimes x \otimes x^\vee \xrightarrow{1_{x^\vee} \otimes ev} x^\vee,
\]
is isomorphic to $1_{x^\vee}$.
\end{definition}

\begin{example}
\begin{itemize}
\item An object $V$ in $\Vect$ (considered as a 2-category with only identity 2-morphisms) is dualizable if and only if it is finite dimensional. 
\item Every object $A$ in $\Alg$ is dualizable. The dual is given by the opposite algebra $A^{op}$. The $1$-morphisms $ev: \C \to A \otimes A^{op}$ and $coev: A^{op}\otimes A \to \C$ in $\Alg$ are both given by the bimodule $A$.
\end{itemize}
\end{example}

\begin{definition}
An object $x$ of $\cC$ is \emph{fully dualizable} if it is dualizable, and the morphism $ev$ admits a left and a right adjoint (i.e. we have morphisms $ev ^L , ev^R: 1_\cC \xrightarrow{} x\otimes x^\vee$, satisfying the usual adjunction properties). 
\end{definition}

\begin{definition}\label{definitionserre}
Given a fully dualizable object in $\cC$, there are canonical maps $S,T:x\to x$, such that
\[
ev^R \simeq (S\otimes id_{x^\vee}) \circ coev,
\]
and
\[
ev^L \simeq (T \otimes id_{x^\vee}) \circ coev,
\]
(here, and for the remainder of the paper we freely use the symmetric structure to identify $x\otimes x^\vee$ with $x^\vee \otimes x$). The map $S$ is called the \emph{Serre automorphism}, and $T$ is inverse to $S$.
\end{definition}

\begin{example}[\cite{FHLT} Example 2.8, \cite{SP} A.3]\label{example2.8}
Suppose $A$ is an object of $\Alg$; we denote by $A^e$ the algebra $A \otimes A^{op}$ (which can be canonically identified with $A^{op} \otimes A$ using the symmetric monoidal structure). We will denote by $A^\ast$, the linear dual $\Hom_{\C}(A,\C)$, and by $A^!$, the bimodule dual $\Hom_{A^e}(A,A^e)$. These are both $A-A$ bimodules. An object $A$ of $\Alg$  is fully dualizable if and only if it is finite dimensional and semisimple. If this case, $ev^L$ is given by $A^!$ and $ev^R$ is given by $A^\ast$ (both considered as $\C - A^e$-bimodules). The analogous statement also holds for $S\Alg$. Alternatively, one can view $A^\ast$ (respectively, $A^!$) as morphisms $A \to A$; as such they are identified with the Serre automorphism $S$ (respectively, inverse Serre automorphism $T$) from Definition \ref{definitionserre}. An analogous statement holds for superalgebras: we write $A^\ast := \HOM_\C(A, \C)$, and $A^! := \HOM_{A^e}(A,A^e)$. 
\end{example}

\subsection{Oriented TQFTs and Frobenius algebras}
The following result indicates a fundamental relationship between TQFT and full dualizability.
\begin{proposition}\label{propositionfullydualizable}
The object $pt_+$ represented by a single point in the \emph{oriented} bordism category $\Bord^{Or}$ is fully dualizable with dual $pt_-$.\
\end{proposition}
\begin{proof}[Proof (sketch)]
The evaluation map
$
ev: pt_- \sqcup pt_+ \to \emptyset^0,
$
and coevaluation map
$
coev: \emptyset^0 \to pt_+ \sqcup pt_-,
$
are both represented by the line interval $[0,1]$. The proof of the duality statement is sometimes referred to as Zorro's Lemma; see Figure \ref{figurecoevandev}. 
\begin{figure}[h]
\caption{Duality data for $pt_+$ in $\Bord^{Or}$}\label{figurecoevandev}
\begin{tikzpicture}

\node at (0,1) [below] {$ev: pt_- \sqcup pt_+ \to \emptyset^0$};
\tqftev

\begin{scope}[xshift = 4cm]
\node at (0,1) [below] {$coev: \emptyset^0 \to pt_+ \sqcup pt_-$};
\tqftcoev
\end{scope}

\begin{scope}[xshift=7cm]
\node at (1,1) [below] {$id_{pt_+}: pt_+ \to pt_+$};
\tqftid
\end{scope}

\begin{scope}[xshift=7cm, yshift=-2cm]
\node at (1,1) [below] {$id_{pt_-}: pt_- \to pt_-$};
\tqftidminus
\end{scope}

\begin{scope}[yshift=-5cm]
\node at (0,1) {$pt_+ \xrightarrow{coev \sqcup id_{pt_+}} pt_+ \sqcup pt_- \sqcup pt_+ \xrightarrow{id_{pt_+} \sqcup ev} pt_+$};
\tqftcoev

\begin{scope}[yshift=-4cm, xshift = -2cm]
\tqftid
\end{scope}

\tqftid
\begin{scope}[yshift = -2cm]
\tqftev
\end{scope}

\draw node at (-1,-5) {$\sidewayssimeq$};

\begin{scope}[yshift=-6cm,xshift=-2cm]
\node at (1,-1) [above] {$id_{pt_+}: pt_+ \to pt_+$};
\tqftid
\end{scope} 

\begin{scope}[xshift=7cm]
\node at (0,1) {$pt_- \xrightarrow{id_{pt+} \sqcup coev} pt_- \sqcup pt_+ \sqcup pt_- \xrightarrow{ev \sqcup id_{pt_-}} pt_-$};
\tqftev

\begin{scope}[xshift=-2cm, yshift=0cm]
\tqftidminus
\end{scope}

\begin{scope}[xshift=0cm, yshift=-4cm]
\tqftidminus
\end{scope}

\begin{scope}[xshift=0cm, yshift = -2cm]
\tqftcoev
\end{scope}

\draw node at (1,-5) {$\sidewayssimeq$};

\begin{scope}[yshift=-6cm]
\node at (1,-1) [above] {$id_{pt_-}: pt_- \to pt_-$};
\tqftidminus
\end{scope} 

\end{scope}
\end{scope}
\end{tikzpicture}
\end{figure}
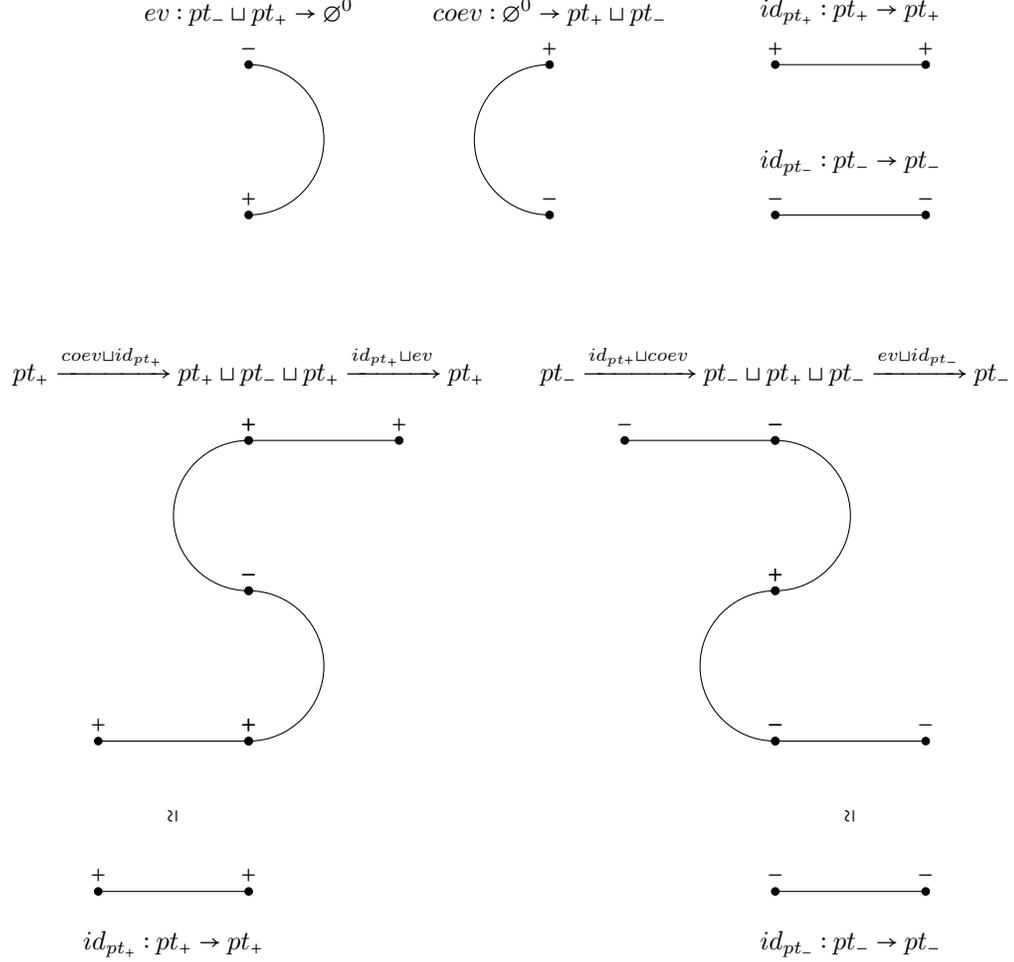

The morphism $coev$ can be identified with the left and right adjoint to $ev$ in $\Bord^{Or}$. The 2-dimensional disc, when considered as a cobordism $\CAP: ev \circ coev \to \emptyset^1$ is the counit, and the when considered as a cobordism $\SADDLE: pt_{+} \sqcup pt_- \to coev \circ ev$ is the unit in the adjunction realizing $coev$ as right adjoint to $ev$ (reading these cobordisms the other way gives the unit and counit of the other adjunction); see Figure \ref{figurefulldualizable}.
\begin{figure}[h]
\caption{Full dualizability data for $pt_+$ in $\Bord^{Or}$}\label{figurefulldualizable}
\begin{tikzpicture}

\begin{scope}[xshift=-3cm]
\node at (-2,0) {$ev \circ coev = S^1$};
\draw [->] (-2, .5) -- (-2,1.5);
\node at (-2,1) [left] {$\CAP$};
\node at (-2,2) {$\emptyset^1$};
\tqftdisc

\begin{scope}[xshift=9cm, yshift=2cm]
\node at (-4.5,-2) {$id_{pt_+} \sqcup id_{pt_-}$}; 
\draw [->] (-4.5, -1.5) -- (-4.5,-0.5);
\node at (-4.5,-1) [left] {$\SADDLE$};
\node at (-4.5,0) {$coev \circ ev$};

\tqftsaddle
\end{scope}
\end{scope}

\begin{scope}[yshift =-4cm]
\node at (-5,-2) {$ev \circ (id_{pt_+ \sqcup pt_-})$};
\draw [->] (-5,-1.5) -- (-5,-.5);
\node at (-5,-1) [left] {$id_{ev} \circ \SADDLE$};
\node at (-5,0) {$ev \circ coev \circ ev$};
\draw [->] (-5,.5) -- (-5,1.5);
\node at (-5,1) [left] {$\CAP \circ id_{ev}$};
\node at (-5,2) {$id_{\emptyset^0} \circ ev$};

\tqftadjunction
\node at (2,-1) {$\simeq$};
\begin{scope}[xshift=3cm]
\node at (3,-2) {$ev$};
\draw [->] (3,-1.5) -- (3,-.5);
\node at (3,-1) [right] {$id_{ev}$};
\node at (3,0) {$ev$};

\tqftidev

\end{scope}
\end{scope}
\end{tikzpicture}
\end{figure}
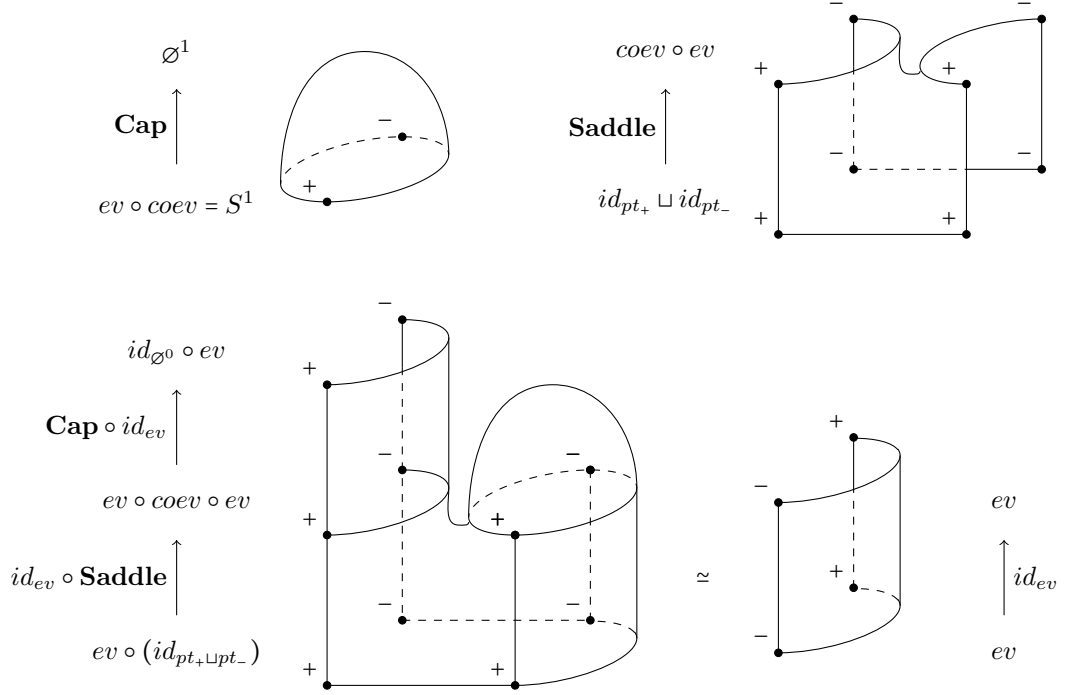

\end{proof}

It follows from Proposition \ref{propositionfullydualizable} that if $Z: \Bord^{Or}$ is an oriented TQFT, then the object $A= Z(pt_+) \in \Alg$ is fully dualizable, and thus is finite dimensional and semisimple. The vector space $Z(S^1)$ can be identified with the abelianization 
\[
\Ab(A) = A \otimes _{A \otimes A^{op}} A = A/[A,A].
\]
Moreover, as $coev$ is both left and right adjoint to $ev$ in $\Bord^{Or}$, there is an isomorphism of bimodules $A^\ast \simeq A$; this isomorphism can be encoded by a non-degenerate inner product on $A$, which factors as 
\[
A\otimes A \to A \to A/[A,A] \xrightarrow{t} \C,
\]
where $t$ is the linear map given by applying $Z$ to the disc cobordism $S^1 \to \emptyset^1$. Thus $A$ has the structure of a Frobenius algebra, and we have isomorphisms 
\[
Z(S^1) \simeq \Ab(A) \simeq \cZ(A)
\]
giving $Z(S^1)$ the structure of a commutative Frobenius algebra. In fact, more is true:

\begin{theorem}[The Cobordism Hypothesis]\cite{SP}\cite{Lur}
Given a finite dimensional, semisimple Frobenius algebra $A$, there is a unique oriented TQFT 
\[
Z_A: \Bord^{Or} \to \Alg
\]
with $Z(pt_+)=A$, and trace map $Z(\CAP): \Ab(A) \to \C$.
\end{theorem}

%

\subsection{1-dimensional spin cobordisms}\label{Spinfacts}
Just as in the oriented case, the $0$-manifold given by a single point has two non-equivalent spin structures. We shall denote these two objects of $\Bord^{Spin}$ by $pt_+$ and $pt_-$ as before. 

Note that there is a canonical involution $\alpha_M$ of every spin manifold $M$, which fixes each point of $M$ but switches the sheets of the spin structure. We can consider this as a spin cobordism with total space $M\times [0,1]$, by using $id_M$ to identify $\{0\}\times M$ with $M$, but $\alpha_M$ to identify $\{1\} \times M$ with $M$.

Thus we have 1-dimensional cobordisms:
\begin{itemize}
\item $id_{pt_\pm}: pt_{\pm} \to pt_{\pm}$,
\item $\alpha_{pt_\pm}: pt_\pm \to pt_\pm$.
\end{itemize}

Unlike in the oriented case, there is no canonical spin cobordism $pt_- \sqcup pt_+ \to \emptyset^0$ which is the evaluation morphism for a duality; rather, there are two choices which can be interchanged by precomposing with $\alpha_{pt_-} \sqcup id_{pt_+}$. However, we will arbitrarily pick one such cobordism:
\[
ev: pt_- \sqcup pt_+ \to \emptyset^0.
\]
Then the cobordism 
\[
coev: \emptyset^0 \to pt_+ \sqcup pt_-,
\]
is defined so that $(ev, coev)$ identify $pt_-$ as dual to $pt_+$ in $\Bord^{Spin}$. We will denote
\[
\overline{ev} := ev \circ (id_{pt_-} \sqcup \alpha_{pt_+}) \simeq ev \circ (\alpha_{pt_-} \sqcup id_{pt_+}).
\]
and 
\[
\overline{coev} := (id_{pt_+} \sqcup \alpha_{pt_-}) \circ coev  \simeq  (\alpha_{pt_+} \sqcup id_{pt_-}) \circ coev.
\]

We have identifications of closed 1-manifolds:
\[
ev \circ coev \simeq \overline{ev} \circ \overline{coev} \simeq S^1_{per}
\]
and
\[
ev \circ \overline{coev} \simeq \overline{ev} \circ {coev} \simeq S^1_{ap}.
\]
Figure \ref{figurespincobordism} illustrates the various one dimensional spin cobordisms and their interactions.

\begin{figure}[h]
\caption{One dimensional spin cobordisms. The spin structures are inherited from the immersion in the plane as illustrated.}\label{figurespincobordism}
\begin{tikzpicture}
\draw node [above] at (1,0.2) {$id_{pt_\pm}$};
\tqftidpm

\begin{scope}[yshift=-2cm]
\draw node [above] at (1,0.2) {$\alpha_{pt_\pm}$};
\tqftalphapm
\end{scope}

\begin{scope}[xshift = 4cm]
\draw node at (0,0.5) {$ev$};
\tqftev
\end{scope}
\begin{scope}[xshift=8cm]
\draw node at (0,.5) {$coev$};
\tqftcoev
\end{scope}

\begin{scope}[xshift=10cm]
\draw node at (1,.5) {$\sigma$};
\tqfts
\end{scope}

\begin{scope}[yshift = -4cm]
\draw node [above] at (1,0.2) {$\alpha_{pt_\pm}$};
\tqftalphapm
\begin{scope}[xshift=2cm]
\draw node [above] at (1,0.2) {$\alpha_{pt_\pm}$};
\tqftalphapm
\end{scope}

\draw node at (1,-1) {$\sidewayssimeq$};

\begin{scope}[yshift=-2cm]
\draw node [above] at (1,0.2) {$id_{pt_\pm}$};
\tqftidpm
\end{scope}

\begin{scope}[xshift=6cm]
\draw node at (0,0.5) {$\overline{ev}$};
\tqftevbar
\end{scope}

\begin{scope}[xshift=10cm]
\draw node at (0,.5) {$\overline{coev}$};
\tqftcoevbar
\end{scope}

\end{scope}

\begin{scope}[yshift=-8cm, xshift=1cm]
\node at (1,0.5) [above]{$ev \circ coev$};

\tqftcoev

\tqfts

\begin{scope}[xshift=2cm]
\tqftev
\end{scope}

\node at (3.5,-1) {$\simeq$};

\begin{scope}[xshift=5cm]
\node at (1,0.5) [above] {$\overline{ev} \circ \overline{coev}$};
\tqftcoevbar

\tqfts

\begin{scope}[xshift=2cm]
\tqftevbar
\end{scope}
\end{scope}

\node at (8.5,-1) {$\simeq$};

\begin{scope}[xshift=9cm]
\node at (1,.5) [above] {$S^1_{per}$};
\tqftsper
\end{scope}

\end{scope}

\begin{scope}[yshift=-12cm, xshift=1cm]
\node at (1,0.5) [above]{$\overline{ev} \circ coev$};

\tqftcoev
\tqfts
\begin{scope}[xshift=2cm]
\tqftevbar
\end{scope}

\node at (3.5,-1) {$\simeq$};

\begin{scope}[xshift = 5cm]
\node at (1,0.5) [above]{$ev \circ \overline{coev}$};

\tqftcoevbar
\tqfts
\begin{scope}[xshift=2cm]
\tqftev
\end{scope}
\end{scope}

\node at (8.5,-1) {$\simeq$};

\begin{scope}[xshift=10cm]
\node at (0,0.5) [above]{$S^1_{ap}$};

\draw (0,-1) circle [radius=1];
\end{scope} 

\end{scope}

\end{tikzpicture}
\end{figure}
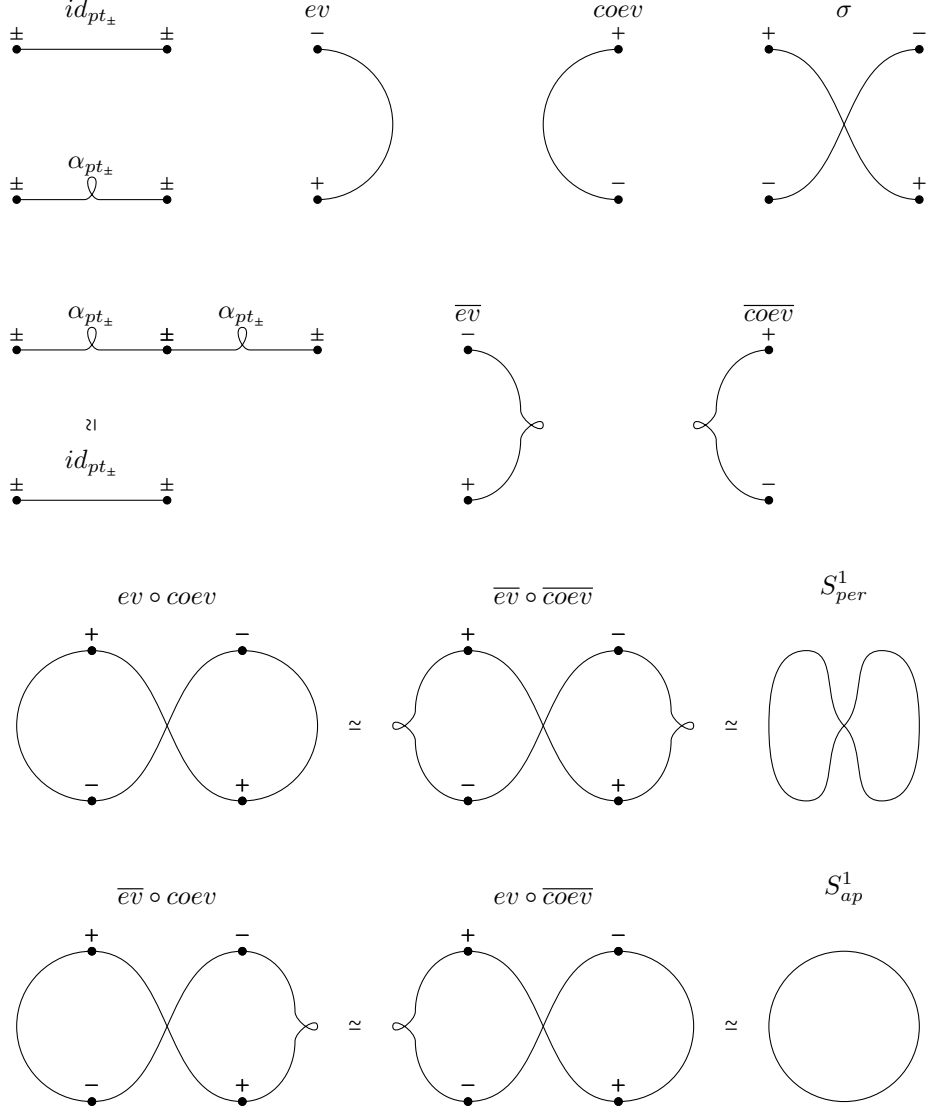

The following result can be seen in the same way as Proposition \ref{propositionfullydualizable}, noting that $ev \circ coev=S^1_{ap}$ is the boundary of a spin disc.
\begin{proposition}
The object $pt_+$ is fully dualizable in $\Bord^{Spin}$. The morphism $\overline{coev}$ is both left and right adjoint to $ev$.
\end{proposition}

\begin{proposition}\label{propositioncircle} 
Let $Z:\Bord^{Spin} \to S\Alg$ be a spin TQFT, and let $A=Z(pt_+)$. Then
\begin{itemize}
\item $Z(S^1_{per}) = \Ab(A)$,
\item $Z(S^1_{ap}) = \cZ(A)$.
\end{itemize}
\end{proposition}
\begin{proof}
Note that the Serre automorphism $\alpha_{pt_+}$ in $\Bord^{Spin}$ is equivalent to its inverse. This defines an isomorphism of bimodules $Z(\alpha_{pt_+}) \simeq A^! \simeq A^\ast$ in $S\Alg$ (see Example \ref{example2.8}).  We have
\begin{align*}
Z(S^1_{per}) = Z(ev \circ coev) = A \otimes _{A^e} A = \Ab(A),\\
Z(S^1_{ap}) = Z(ev \circ \overline{coev}) = A^! \otimes _{A^e} A = \cZ(A),
\end{align*}
as required.
\end{proof}

\subsection{$2$-dimensional spin cobordisms}

The 2-disc carries a unique spin structure with gives rise to cobordisms
\[
\CAP: S^1_{ap} \to \emptyset^1, 
\]
and
\[
\CUP: \emptyset^1 \to S^1_{ap} .
\]

Recall that the surface $S^2$ carries a unique (even) spin structure. Removing 3 discs from $S^2$, we obtain spin cobordisms:
 \[
_{ap,ap}\PANTS_{ap}: S^1 _{ap} \sqcup S^1 _{ap} \xrightarrow{} S^1 _{ap}
\]
 and 
\[
_{ap}\PANTS_{ap,ap}: S^1 _{ap} \sqcup S^1 _{ap} \xrightarrow{} S^1 _{ap}.
\]

We can now build higher genus cobordisms by gluing copies of $_{ap,ap}\PANTS_{ap}$, $_{ap}\PANTS_{ap,ap}$, $\CUP$, and $\CAP$. The punctured spin surfaces obtained in this way will always have even parity.


There are two non-isomorphic spin cobordisms:
\[
id_{S^1_{per}}, \alpha_{S^1_{per}}: S^1_{per} \to S^1_{per},
\]
(note that the involution $\alpha_{S^1_{ap}}$ is isotopic to the identity).
Now let us remove a disc from the cylinder $id_{S^1_{per}}$. If we consider the resulting boundary circle as incoming, we obtain a cobordism:
\[
_{per,ap}\PANTS_{per} : S^1 _{per} \sqcup S^1 _{ap} \xrightarrow{} S^1 _{per}.
\]
If we consider the resulting boundary circle as outgoing, we obtain a cobordism:
\[
_{per}\PANTS _{ap, per}S^1 _{per} \xrightarrow{} S^1 _{ap} \sqcup S^1 _{per} .
\]

Similarly, we can remove a disc from the cylinder $\alpha_{S^1_{per}}$ to obtain cobordisms $_{per,ap}\overline{\PANTS}_{per}$ and $_{per}\overline{\PANTS}_{per, ap}$. 

The two cylinders $id_{S^1_{per}}$ and $\alpha_{S^1_{per}}$ can also be considered as cobordisms
\[
ev_{S^1_{per}}, \overline{ev}_{S^1_{per}}: S^1_{per} \sqcup S^1_{per} \to \emptyset^1,
\]
and
\[
coev_{S^1_{per}}, \overline{coev}_{S^1_{per}}: \emptyset^1 \to S^1_{per} \sqcup S^1_{per}.
\]
As in Section \ref{Spinfacts}, neither one of $ev_{S^1_{per}}$ or $\overline{ev}_{S^1_{per}}$ is distinguished -- they are exchanged by precomposing with $\alpha_{S^1_{per}} \sqcup id_{S^1_{per}}$. An analogous statement holds for $coev_{S^1_{per}}$. By removing a disc from $ev_{S^1_{per}}$ we obtain a cobordism
\[
_{per,per}\PANTS_{ap}S^1 _{per} \sqcup S^1 _{per} \xrightarrow{} S^1 _{ap}.
\]
We define cobordisms $_{per,per}\overline{\PANTS} _{ap}$, $_{ap}\PANTS_{per,per}$, and $_{ap}\overline{\PANTS}_{per,per}$ similarly. The notation is chosen such that the pairs $(ev_{S^1_{per}}, coev_{S^1_{per}})$ and $(\overline{ev}_{S^1_{per}}, \overline{coev}_{S^1_{per}})$ both define evaluation and coevalution maps which exhibit $S^1 _{per}$ as its own dual. 

Let 
\[
E_{even}: S^1_{ap} \to S^1_{ap}
\]
denote the cobordism obtained by removing two discs from a surface of genus 1 with an even spin structure. Similarly, let $E_{odd}$ denote the analogous cobordism where the genus 1 surface has an odd spin structure. The following lemmas may easily be checked using (for example) Johnson's computation of the parity \cite{Jo}.

\begin{lemma}
We have:
\begin{align*}
_{per,per}\PANTS_{ap} \circ {_{ap}\PANTS}_{per,per} \simeq _{per,per}\overline{\PANTS}_{ap} \circ {_{ap}\overline{\PANTS}}_{per,per} \simeq E_{even} \\
_{per,per}\overline{\PANTS}_{ap} \circ {_{ap}\PANTS}_{per,per} \simeq _{per,per}\PANTS_{ap} \circ {_{ap}\overline{\PANTS}}_{per,per} \simeq E_{odd}
\end{align*}
\end{lemma}

\begin{lemma}
The composite
\[
\CAP \circ E_{even} ^{\circ g} \circ \CUP
\]
is an even spin surface of genus g. The composite
\[
\CAP \circ E_{odd} \circ E_{even} ^{\circ (g-1)} \circ \CUP,
\]
is an odd spin surface of genus g.
\end{lemma}

\subsection{Computing invariants of spin TQFT}

Let $Z: \Bord^{Spin} \to S\Alg$ be a spin TQFT and denote the superalgebra $Z(pt_+)$ by $A$. By Wedderburn theory (see Section \ref{Sergeev}), we may assume that $A$ takes the form:
\[
A = \prod _{S_0} \C \times \prod _{S_1} \Cl_1,
\]
where $S= S_0 \sqcup S_1$ is the set of simple supermodules of $A$ decomposed into type $M$ and type $Q$ parts. Write $\deg (s)=i$ whenever $s\in S_i$. Combining Proposition \ref{propositioncircle} with the results of Section \ref{Sergeev}, we obtain:

\begin{proposition}
A spin TQFT $Z$ as above assigns the following invariants to a spin circle:
\begin{itemize}
\item $Z(S^1_{ap}) = \cZ (A) \cong \C^{|S|}$ which has a basis of orthogonal idempotents $e_s$, (all in degree $0$).
\item $Z(S^1_{per}) = \Ab(A) \cong \C ^{|S_0|} \oplus \C^{|S_1|}[1]$.
\end{itemize}
\end{proposition}

Consider the linear map
\[
t:= Z(\CAP): Z(S^1 _{ap}) \to \C,
\]
and write $t(s)= t(e_s)\in \C$ for each $s\in S$. The goal of this section is to derive a formula for the value of $Z$ on surfaces in terms of the data of $S=S_0 \sqcup S_1$ and the numbers $t(s)$.

\begin{proposition}\label{inv}
Let $\Sigma$ be a closed spin surface with $k$ marked points considered as a cobordism $\left( S^1 _{ap}\right)^{\sqcup k} \to \emptyset^1$. Then we have
\[
Z(\Sigma): e_{s_1} \otimes \ldots \otimes e_{s_k} \mapsto 
\begin{cases}(-1)^{\deg(s_1) p(\Sigma)} t(s_1)^{\chi(\Sigma)/2} \quad \text{if $s_1 = \ldots = s_k$} \\ $0$ \quad \text{otherwise.} \end{cases}
\]
In particular, $Z$ assigns $\sum _{s\in S} (-1)^{\deg(s)p(\Sigma)}t(s)^{\chi(\Sigma)/2}$ to a closed spin surface. 
\end{proposition}

\begin{proof} 
The result for even surfaces follows from the usual techniques of oriented 2d TQFT. Explicitly, the  map $Z(_{ap,ap}\PANTS_{ap})$ recovers the multiplication map on $Z(S^1_{ap}) = \cZ(A)$. The cobordism $\CAP$ defines a non-degenerate trace, giving $\cZ(A)$ the structure of a commutative Frobenius algebra (in degree 0). It is an easy exercise in the theory of semisimple Frobenius algebras that the comultiplication map $Z(_{ap}\PANTS_{ap,ap})$ is given by $e_s \mapsto t(s)^{-1} e_s \otimes e_s$. It follows that 
\[
Z(E_{even}) = Z(_{ap,ap}\PANTS_{ap}) \circ _{ap}\PANTS_{ap,ap}: e_s \mapsto t(s)^{-1} e_s.
\]
The even spin surface of genus $g$ with $k$ punctures can be factorized as $\CAP \circ E_{even} ^g \circ _{ap^k}\PANTS_{ap}$, where $_{ap^k}\PANTS_{ap}$ is the pair of pants with $k$ incoming boundary circles. This gives the result in the even case.

In order to complete the computation for odd surfaces, it remains to compute $Z(E_{odd})$. First we observe the following:

\begin{lemma}\label{lemmagrading}
The linear map $Z(\alpha_{S^1_{per}}): \Ab(A) \to \Ab(A)$ is given by the grading involution.
\end{lemma}
\begin{proof}
It is sufficient to prove the lemma in the cases $A= \C$ and $A= \Cl_1$. Note that the cylinder $\alpha_{S^1_{per}}$ can be thought of as a product $S^1_{per} \times \alpha_{pt_+}$. Thus the map $Z(\alpha_{S^1_{per}})$ is given by the automorphism of $\Ab(A)$ induced from the $A-A$ bimodule $Z(\alpha_{pt_+})$ (in general, any Morita equivalence of algebras induces a trace map on abelianizations). This bimodule is necessarily the Serre automorphism of $A$. In the case $A=\C$, the Serre automorphism and grading involution are trivial. In the case $A=\Cl_1$, the Serre automorphism is the Shift bimodule $A[1]$; the trace of the shift bimodule is precisely the grading automorphism, as required.
\end{proof}
\begin{remark}
Often it is included in the axioms of spin TQFT that the canonical spin automorphism of a 1-manifold gives rise to the grading automorphism on the corresponding super vector space. For us this follows from the fact that the spin automorphism of a point is also the Serre automorphism in $\Bord ^{spin}$ as shown in Lemma \ref{lemmagrading}.
\end{remark}

The map 
\[
Z(_{per, ap}\PANTS_{per}): Z(S^1_{ap}) \otimes Z(S^1_{per}) \to Z(S^1_{per})
\]
recovers the action canonical action of $Z(A)$ on $\Ab(A)$. The elements $e_s \in Z(A)$ give a family of orthogonal idempotent operators on $\Ab(A)$. The eigenspaces $I_s$ of the operators $e_s$ are one dimensional and generate $\Ab(A)$. 

The super vector space $\Ab(A)$ can be equipped with the inner product given by$Z(ev_{S^1_{per}})$. The eigenspaces $I_s$ are mutually orthogonal with respect to this inner product, and we can choose a basis $f_s$ of $\Ab(A)$ which is orthonormal with respect to the inner product (such a basis is only canonical up to sign, as we could have picked the inner product $Z(\overline{ev}_{S^1_{per}})$. Note that $\deg(f_s) = \deg(s)$.

Now, by construction,
\[
Z(_{per,per}\PANTS_{ap}): f_{s_1} \otimes f_{s_2} \mapsto \begin{cases}
e_{s} \quad \text{if $s_1 = s_2 = s$},\\
0 \quad \text{if $s_1 \neq s_2$}.
\end{cases}
\]
and
\[
Z(_{ap}\PANTS_{per,per}): e_s \mapsto f_s \otimes f_s.
\]

As $_{per,per}\overline{\PANTS}_{ap}$ is given by the composition $_{per,per}{\PANTS}_{ap} \circ \left( id_{S^1_{per}} \sqcup \alpha_{S^1_{per}} \right)$, we also have 
\[
_{per,per}\overline{\PANTS}_{ap}: f_s \otimes f_s \mapsto (-1)^{\deg(s)} e_s.
\]
Finally, $E_{odd}$ is given by the composition 
\[
Z(_{per,per}\overline{\PANTS_{ap}}) \circ Z(_{ap}\PANTS_{per,per}), 
\]
and thus $Z(E_{odd})$ maps $e_s$ to $(-1)^{\deg(s)} e_s$. This completes the proof of Proposition \ref{inv}.
 
\end{proof}

\subsection{Classification of 2d spin TQFTs}
Proposition \ref{inv} computes the invariants of a TQFT $Z$ such that $Z(pt)$ is a finite dimensional semisimple algebra $A$, starting from the linear map $t: \cZ (A) \to \C$ given by applying $Z$ to a spin disc. In fact, the cobordism hypothesis \cite{Lur} implies that any TQFT is completely determined by such data, and that, given the data of a semisimple algebra $A$ and a map $t$ satisfying certain conditions, there is a TQFT which assigns those data to a point and a disc. \footnote{A classification of Spin (2,1) TQFTs (i.e. non-extended TQFTs) was described in \cite{DBMS}.}

\begin{proposition}\label{class}
A 2d spin TQFT is given by a semisimple superalgebra $A$, together with a linear map $t: \cZ(A) \to \C$ such that the composite
\[
A^! \otimes A \to A^! \otimes _{A^e} A = \cZ(A) \xrightarrow{t} \C
\]
identifies $A^!$ with the linear dual of $A$. This is equivalent to the following data: a finite set $S = S_0 \sqcup S_1$ (the spectrum of $\cZ (A)$ decomposed into type $M$ and $Q$), together with a function $t: S = S_0 \sqcup S_1 \to \C^\times$ (the trace map on $\cZ(A)$ restricted to the orthogonal idempotents). 
\end{proposition}
\begin{proof}[Proof (sketch)]
According to the cobordism hypothesis, a spin TQFT should be given by a fully dualizable object of $S\Alg$, equipped with the structure of a $Spin(2)$ homotopy fixed point. The fully dualizable object determines a \emph{framed} theory, and the fixed point data allows an extension to spin manifolds.

An algebra $A$ is fully dualizable if and only if it is finite dimensional and semisimple (see Example \ref{example2.8}), so by Wedderburn theory $A$ is Morita equivalent to a product $\prod _{S_0} \C \times \prod _{S_1} \Cl_1$.

The action of $SO(2)$ on the space of fully dualizable objects gives an automorphism of each object (the Serre automorphism), which in the case of algebras is the bimodule dual $A^!$ of $A$ (or it's inverse $A^\ast$). Similarly, the action of $Spin(2)$ on this space gives the bimodule $A^! \otimes _A A^!$.

Let us first identify the space of $SO(2)$ fixed points (corresponding to oriented TQFTs). Giving the structure of a homotopy fixed point for $SO(2)$ is to give an isomorphism  of bimodules $A \cong A^!$. Unwinding the definitions of dual bimodule give that this is the same as a linear map $\Ab(A) = A \otimes _{A\otimes A^{op}} A \to \C$ such that the composite
\[
A\otimes A \xrightarrow{m} A \otimes _{A^e} A \xrightarrow{tr} \C
\]
identifies $A$ with its linear dual $A^\ast$. This is the structure of a symmetric Frobenius algebra on $A$ \cite{Lur}, \cite{SP}. The centre and abelianization of $A$ are both identified with the vector space with basis $e_s$, $s\in S$. A map $tr$ as above is exactly given by a function $t:S \to \C ^\ast$.

Now, to give a $Spin(2)$ fixed point is to give a bimodule isomorphism $A^! \cong A^\ast$. Again, unwinding the dualities gives us a map $\cZ(A) = A^! \otimes _{A^e} A \to \C$, such that the composite $A^! \otimes A \to A^! \otimes _{A^e} A \to \C$ identifies $A^!$ with the linear dual of $A$. This is equivalent to a function $S \to \C ^\ast$ as required.
\end{proof}

\begin{remark}
If the 2-category $S\Alg$ was replaced by an $(\infty ,2)$ category with non-trivial higher morphisms (for example, the category of differential graded algebra, bimodules, maps of bimodules, homotopies between maps etc...), there would be higher coherence data to consider. In particular, specifying a $SO(2)$ or $Spin(2)$ (homotopy) fixed point would involve more data.
\end{remark}

%

\begin{remark}
If $A$ contains only factors of type $M$ (i.e. $S_1 = \emptyset$), then the corresponding TQFT cannot distinguish spin structures. It follows that a spin TQFT valued in ordinary (ungraded) algebras necessarily comes from an oriented theory. 
\end{remark}

\begin{example}[The Parity TQFT]\label{AtTh}
Consider the spin TQFT $Z_1$ for which $S_0 = \emptyset$ and $S_1 = \{\ast \}$ and $t(\ast )=1$. This assigns $\Cl_1$ to a point, the vector space $\C$ to $S^1 _{per}$, $\C[1]$ to $S^1 _{per}$, and the number $(-1)^{p(\Sigma)}$ to every spin surface $\Sigma$. A direct construction of $Z_1$ is given in Section \ref{AtiyahTheory}.
\end{example}

\section{Averaging TQFTs over finite covers}\label{average}  

The aim of this section is to give a procedure for taking a TQFT $Z_1$ and producing a new TQFT $Z_n$ for each positive integer $n$ which averages $Z_1$ over $n$-fold covers. This procedure is based on the machinery of finite path integrals as defined in \cite{FHLT} (based on ideas from \cite{Freed}). See also the recent paper \cite{Morton} \footnote{See also the entry for \cite{FHLT} at ncatlab.org, written by Urs Schreiber, as well as \cite{Lur}.} 

For example, applying this procedure to the trivial TQFT recovers the theory which assigns the ordinary Hurwitz numbers to a closed 2-manifold. Applying this to the parity theory (see Example \ref{AtTh}) will produce the spin Hurwitz theory of Theorem \ref{TFT}.


\begin{notation} 
In this section, we will write $\Bord$ instead of $\Bord^{Spin}$, $\Alg$ instead of $S\Alg$ and $\Vect$ instead of $S\Vect$. The results of this section apply much more generally. For example $\Bord$ could be the oriented or framed bordism category (or $(\infty ,2)$-category), and $\Alg$ could be the Morita category of ordinary algebras, or the $(\infty ,2)$-category of dg-algebras.

If $\cC$ is a symmetric monoidal 2-category, write $\Omega \cC$ for the symmetric monoidal 1-category of endomorphisms of the identity object $1_\cC$. Similarly, write $\Omega ^2 \cC$ for the monoid of endomorphisms of $1_{\Omega \cC}$. We call objects of $\cC$ $0$-objects, objects of $\Omega \cC$ $1$-objects and elements of $\Omega ^2 \cC$ $2$-objects. For example, if $\cC = \Bord$ then $n$-objects are closed $n$-manifolds.
\end{notation}

\subsection{The averaging theory $Z_n$}.
Let $Z_1: \Bord \to \Alg$ be a TQFT. The goal of this section is to prove the following
\begin{proposition}\label{CoveringTFT}
Given $Z_1$ as above, there is a family of TQFTs $Z_n: \Bord \to \Alg$ for $n=2,3, \ldots$, such that
\begin{itemize}
\item $Z_n (pt) = Z_1(pt)^{\otimes n} \rtimes S_n$.
\item If $N$ is a closed 1-manifold, $Z_n(N) = \bigoplus Z_1(\widetilde{N})^{\Aut(\widetilde{N}/N)}$, where the sum is over isomorphism classes of $n$-fold covers $\widetilde{N}$ of $N$.
\item If $\Sigma$ is a cobordism between closed 1-manifolds $N_0$ and $N_1$, $Z_n (\Sigma)$ is given by the following linear map: 
\begin{align*}
\bigoplus  Z_1(\widetilde{N_0})^{\Aut(\widetilde{N_0}/N_0)} &\xrightarrow{} \ \bigoplus Z_1(\widetilde{N_1})^{\Aut(\widetilde{N_1}/N_1)} \\
u = \left( u_{\widetilde{N_0}} \right) &\mapsto v= \left( v_{\widetilde{N_1}} \right)
\end{align*}
where 
\[
v_{\widetilde{N_1}} = \sum \frac{ Z_1(\widetilde{\Sigma})(u)}{\# \Aut(\widetilde{\Sigma} / \Sigma)}
\]
and the sum is over isomorphism classes of covers $\widetilde{\Sigma} / \Sigma$.
\end{itemize}
\end{proposition}

The field theory $Z_n$ will be defined as a composite
\begin{equation*}\label{composite}
\Bord  \xrightarrow{\Cov _n} \Fam _2 (\Bord ) \xrightarrow{\Fam _2(Z)} \Fam _2 (\Alg ) \xrightarrow{\Sum_2} \Alg.
\end{equation*}

Very loosely, the map $\Cov _n$ takes a manifold to its groupoid of $n$-fold covers (whilst remembering the spin structure on the total space of each cover), the functor $\Fam _2(Z)$ applies the TQFT $Z_1$ to the total space of the cover, and then $\Sum _2$ takes the average. Most of the categories and functors appearing above are defined in section 3 of \cite{FHLT} (see also \cite{Lur}). Here, we will briefly review the definitions (note that we only need the case of $2$-categories, rather than the general $m$-categories discussed in \cite{FHLT}). 

\subsection{Toy example: 1d TQFTs}
Before we delve into the details of the construction, let us examine what it gives us in the case of a 1d TQFT. Recall that a 1d (say, oriented) TFT $Z$ valued in vector spaces is determined by a finite dimensional vector space $V$ which is the value at a (positively oriented) point. The linear maps $\C \to V \otimes V^\ast$ and $V^\ast \otimes V \to \C$ which $Z$ assigns to semicircles are necessarily the unit and trace maps which identify $V^\ast$ as the dual of $V$. Hence the value of $Z$ on a circle is given by the integer $d := \dim V = tr(1_V)$. 

The analogue of Proposition \ref{CoveringTFT} says that we can define a new TFT $Z_n$ whose value at a point is $( V ^{\otimes n})^{S_n}$ and whose value on a circle is $\sum _{\bmu } \frac{d^{\ell (\bmu)}}{\left| C(\bmu) \right| }$. Here, the sum is taken over partitions $\mu$ which index $\pi _0 (\Cov (S^1))$, and $C_(\bmu)$ is the centralizer of the corresponding conjugacy class in $S_n$. 

By comparing the results from Propositions \ref{CoveringTFT} with what we know by duality, we recover the formula: 
\[
\dim ( V ^{\otimes n})^{S_n} =  \sum _{\bmu } \frac{d^{\ell (\bmu)}}{\# C(\bmu) }.
\]
The formulas in Theorem \ref{Hurwitz} will be proved using a very similar idea.

\subsection{The category $\Fam _2(\cC)$} 
Suppose $\cC$ is a symmetric monoidal 2-category. The symmetric monoidal 2-category $\Fam _2(\cC)$ has objects essentially finite groupoids $X$ (i.e. $\pi_0(X)$ and $\pi _1 (X)$ are finite) with a functor $f: X \to \cC$. The 1-morphisms are spans of groupoids over $\cC$, i.e. a groupoid $W$, with maps $p_1$, $p_2$ to $X$ and $Y$ and a natural transformation $\alpha : f_1 \circ p_1 \Rightarrow f_2 \circ p_2$:
\begin{equation}\label{span}
\xymatrix
{
&  W \ar[ld]_{p_1} \ar[rd]^{p_2} &\\
 X \ar[rd]_{f_1} \ar@{=}[r] & \alpha \ar@{=>}[r] &  Y \ar[ld]^{f_2} \\
& \cC &
}
\end{equation}

Composition of morphisms is given by the homotopy fibre product of groupoids. Similarly, the 2-morphisms are given by spans of such diagrams (up to equivalence), and the symmetric monoidal structure is given by products of groupoids. The details of how 2-morphisms compose etc. quickly become quite complicated and we will not attempt to flesh them out here.We refer the reader to \cite{Morton} for more details.
\begin{remark}\label{2vect}
Note that the $2$-category of \emph{$2$-vector spaces} defined in \cite{Morton} is equivalent to the full subcategory of $\Alg$ given by finite dimensional semisimple algebras: to such an algebra $A$, one assigns its category of finite dimensional modules which is a $2$-vector space.
\end{remark}

However, the category $\Omega\Fam _2 (\cC) = \Fam _1 (\Omega \cC)$ is easier to describe: it's objects are essentially finite groupoids $X$ with a map $f:X \to \Omega \cC$. Morphisms are spans of groupoids as in the diagram \ref{span}, except that $\cC$ is replaced by $\Omega \cC$. Similarly, $\Omega ^2 \Fam _2 (\cC) = \Fam _0 (\Omega ^2 \cC)$ is the monoid of equivalence classes of groupoids $X$ with a map $\pi _0 (X) \to \Omega ^2 \cC$. 

Note that if we have a symmetric monoidal functor of 2-categories $F:\cC \to \cD$, then we get a symmetric monoidal functor $\Fam_2(F): \Fam _2 (\cC) \to \Fam _2 (\cD)$ by composing all the maps with $F$.

\subsection{The functor $\Cov _n$}
If $M$ is a topological space, let $\Cov _n(M)$ denote the groupoid whose objects are $n$-fold covers of $M$, and morphisms are deck transformations. The functor $\Cov _n : \Bord  \to \Fam _2 (\Bord )$ takes a closed manifold $M$ to $\Cov _n (M)$ equipped with the map $\Cov _n (M) \to \Bord $ which takes a covering space to the its total space thought of as an object of $\Bord$. To a cobordism $M_0 \xrightarrow{N} M_1$ between closed manifolds, we associate the span
\[
\Cov _n (M_0) \longleftarrow \Cov _n (N) \longrightarrow \Cov _n (M_1)
\]
equipped with the maps $\Cov _n (M_i) \to \Bord $. 

\subsection{The functor $\Sum _2$}\label{Sum2}
Let us first describe what $\Sum _2 : \Fam _2 (\Alg) \to \Alg$ does on $i$-objects, $i=0,1,2$:
\begin{itemize}
\item To a 2-object $(X, f: \pi _0 (X) \to \C)$ of $\Fam _2 (\Alg)$, $\Sum _2$ assigns the number
\[
\sum _{x \in \pi _0 (X)} \frac{f(x)}{\#\Aut (x) }.
\]
\item To a 1-object $(X, f:X \to \Vect)$ of $\Fam _2 (\Alg)$, $\Sum _2$ assigns the limit of the diagram $f$ which can be identified with
\[
\bigoplus _{x \in \pi _0 (X)} f(x)^{\Aut (x)}.
\]
\item To a 0-object $(X, f:X \to \Alg)$, $\Sum _2$ assigns the (homotopy) limit of the diagram $f$ which can be identified with\footnote{The computation of this limit is given in \cite{Morton} in the context of $2$-Vector spaces (see Remark \ref{2vect}.} 
\[
\bigoplus _{x\in \pi _0 (X) } f(x) \rtimes \Aut (x).
\]
\end{itemize}

We would like to describe how $\Sum _2$ acts on more general morphisms in $\Fam _2 (\Alg)$. This will be given by an integral transform formula as follows, which we will describe in the case of a morphism in $\Omega \Fam _2 (\Alg) = \Fam _1 (\Vect)$.

%

First let us note that if $(X, f:X \to \Vect)$ is a 1-object, then the  natural map 
\[
\lim f = \bigoplus _{x\in \pi _0}  f(x) ^{\Aut (x)} \too \bigoplus _{x\in \pi _0}  f(x) _{\Aut (x)} = \colim f
\]
is an isomorphism (where the subscript denote coinvariants). The inverse is given by lifting $v \in f(x)_{\Aut (x)}$ to $\widetilde{v} \in f(x)$ and then mapping to 
\[
\sum _{g \in \Aut (x)} g.\widetilde{v} \in f(x) ^{\Aut (x)}.
\]

Now, given a sequence $W \stackrel{p}{\to} X \stackrel{f}{\to} \Vect$, as well as the natural pullback map $f^\ast : \lim (f) \to \lim (fp)$ we have the pushforward $f_\ast : \lim (fp) \to \lim (f)$, defined by first identifying limits with colimits, then using the natural pushforward map for colimits. 

Recall that a morphism of 1-objects is given by a span of groupoids as in diagram \ref{span} (where $\cC = \Vect$). In particular, for each $w\in \pi _0(W)$, we have a linear map $\alpha (w): f_1 (p_1(w)) \to f_2 (p_2 (w))$ which is $\Aut (w)$ equivariant. Then $\Sum _2 (W,p_1 ,p_2, \alpha)$ is given by the integral transform formula:
\[
\lim p_1 \xrightarrow{p_1 ^\ast} \lim f_1 p_1 \xrightarrow{\alpha} \lim f_2 p_2 \xrightarrow{p_{1\ast}} \lim p_2 .
\]

Explicitly, this is given by the linear map 
\[
\bigoplus _{x\in \pi _0 X}  f_1(x) ^{\Aut {x}} \xrightarrow{} \bigoplus _{y\in \pi _0 Y}  f_2(y) ^{\Aut (y)}
\]
where $u = (u_x)$ gets mapped to $v=(v_y)$ and
\begin{align*}\label{inttrans}
v_y = \sum _{w\in \pi _0 W_y} \frac{\alpha (w)(u)}{|\Aut (w)|}.
\end{align*}

The fact that this construction is functorial follows from a base change formula for the pullback and pushforward: $p_2 ^\ast p_{1\ast} \cong \widetilde{p}_{1 \ast} \widetilde{p}_2 ^\ast$ whenever $p_1, p_2, \widetilde{p}_1 , \widetilde {p}_2$ form a cartesian square.

The same integral transform idea allows us to define $\Sum _2$ on 1-morphisms; the description for general 2-morphisms is somewhat more complicated (again, see \cite{Morton} for a careful construction of the functor $\Sum_2$).

Proposition \ref{CoveringTFT} now follows from the desription of the functor $\Sum_2$ in Section \ref{Sum2}.

%

\section{Spin Hurwitz Numbers}\label{finalsection}
In this section we will combine the results from Sections \ref{SpinTFT} and \ref{average} to construct the spin Hurwitz theory of Theorem \ref{TFT} and prove the spin Hurwitz formula of Theorem \ref{Hurwitz}.

\subsection{The spin Hurwitz theory}\label{proofoftft}
The TQFT $Z_n$ is obtained by applying Proposition \ref{CoveringTFT} in the case when $Z_1$ is the parity theory (see Example \ref{AtTh}). It is clear from the proposition that the value of $Z_n$ on a closed spin surface $\Sigma$ is the spin Hurwitz number $\cH_n(\Sigma)$. 

To complete the proof of Theorem \ref{TFT} we should also identify $Z_n(S^1_{ap})$ with $\C[\OP(n)]$. Here we want more than just an identification as abstract vector spaces, rather, we want to identify a particular basis $\delta_\bmu$, $\bmu \in \OP(n)$ such that 
\begin{equation}\label{delta}
\cH_n(\Sigma, \bmu^1, \ldots, \bmu^k) = Z_n(\Sigma^\ast)(\delta_{\bmu^1} \otimes \ldots \otimes \delta_{\bmu^k}),
\end{equation}
where $\Sigma^\ast$ is the cobordism from $(S^1_{ap})^{\sqcup k}$ to $\emptyset^1$ obtained by puncturing $\Sigma$ at $k$ points.

Let $\bmu \in \PP(n)$ be a partition corresponding to a covering $\bigsqcup _i \cS_{\mu _i} \to \cS_{ap}^1$ (so each $\cS_{\mu _i} \to S^1$ is a connected cover of degree $\mu _i$); the automorphism group of this covering is the centralizer $C(\bmu)$ in $S_n$. 

If one of the $\mu _i$ is even, then the corresponding component of the cover will be periodic. There is a unique non-trivial order 2 automorphism of the cover, fixing all the components except $\cS_{\mu _i}$ on which it acts by switching the sheets of the spin structure. Thus it acts by $-1$ on the odd vector space $Z_1(\cS_{\mu_i})$, so that the invariants of the action of $C(\bmu)$ on $Z_1(\cS_\bmu)$ vanish. On the other hand, if $\bmu$ is an odd partition, then the action on $Z_1(\cS_\bmu)$ is trivial. Thus 
\begin{align}
Z_n(\cS_{ap}^1) = \bigoplus _{\bmu \in \OP(n)} Z_1(\cS_{\mu_1}) \otimes \ldots \otimes Z_1(\cS_{\mu_{\ell(\bmu)}}).
\end{align}
As $\cS_{\mu_i} \simeq S^1_{ap}$, there is spin disc bounding $\cS_{\mu_i}$, and thus $Z_1(\cS_{\mu_i})$ contains a canonical element, $1_{\mu_i}$. The element $\delta_\bmu$ is defined to be
\[
1_{\mu_1} \otimes \ldots \otimes 1_{\mu_{\ell(\bmu)}} \in Z_1(\cS_{\mu_1}) \otimes \ldots \otimes Z_1(\cS_{\mu_{\ell(\bmu)}}) = Z_1(\cS_{\bmu}) \subseteq Z_n(S^1_{ap}).
\]
The property in Equation \ref{delta} follows from Proposition \ref{CoveringTFT}. 

\begin{remark}
The statement that only odd partitions appear as boundary conditions in out TQFT can be thought of more concretely. Suppose you have a decomposition of a spin surface $\Sigma$ into two halves along a boundary circle. We wish to construct covers of the surface by gluing together covers along the two halves. If we have a cover on one of the halves with non-odd ramification data at the boundary, one of the components of the boundary of the cover will be periodic. There are always two ways to glue this circle to a corresponding one over the other half, which lead to spin surfaces covering our original surface \emph{which have opposite parity}. Thus, when we sum over all covers, these will cancel and will not contribute to the sum.
\end{remark} 

Now let us investigate $Z_n(S^1_{per})$ (this will not be necessary for the spin Hurwitz formulas). In this case, all components of the covering $S_\bmu$ are periodic. Given a partition $\bmu$ for which $\mu_i = \mu _j$, we have an automorphism which switches the two components. As in the anti-periodic case, this acts by $-1$ on $Z_1 (\cS_\bmu)$ so is killed in the space of invariants. On the other hand if $\bmu$ is a strict partition, then $Z_1(\cS_\bmu)$ is given by
\[
Z_1(\cS_{\mu_1}) \otimes \ldots \otimes Z_1(\cS_{\mu_{\ell(\bmu)}})
\]
which is a super vector space of degree $\ell(\bmu) \ \mod 2$. This completes the proof of Theorem \ref{TFT}.

\begin{remark}
The TQFT point of view allows us to generalize the notion of spin Hurwitz number to allow periodic boundary components. One has to be careful though, as the odd line $L=Z_1 (S^1_{per})$ is not trivialized; rather, it has a canonical element defined up to sign. This means that we should interpret such invariants as linear maps between linear combinations of copies of $L$ rather than numbers.
\end{remark} 
 
\subsection{The spin Hurwitz formulas}
Theorem \ref{TFT} states that the spin Hurwitz numbers $\cH_n(\Sigma, \bmu^1,\ldots, \bmu^k)$ are given by evaluating $Z_n(\Sigma^\ast)$ on the element $\delta_{\bmu^1} \otimes \ldots \otimes \delta_{\bmu^k} \in \cZ(\cY_n)$. On the other hand, Proposition \ref{inv} computes the invariant obtained by evaluating $Z(\Sigma^\ast)$ on the elements of the form $e_{\bnu^1} \otimes \ldots \otimes e_{\bnu^k}$ in terms of the numbers $t_\bnu = Z_n(\CAP)(e_\bnu)$. Note that 
\begin{equation}\label{equationt}
Z_n(\CAP)(\delta_\bmu) =
\begin{cases}
1/n! \quad \text{if $\bmu = 1^n$}\\
0 \quad \text{otherwise}
\end{cases}
\end{equation}
Thus, to prove Theorem \ref{Hurwitz} (the spin Hurwitz formulas), it remains to identify the change of basis matrix between the bases $\{\delta_\bmu \mid \bmu \in \OP(n)\}$ and $\{e_\bnu \mid  \bnu \in \SP(n)\}$.

Recall that $\cY_n$ is canonically identified with the twisted group algebra of the hyperoctahedral group $B_n$. Thus there is another basis $\{ \Delta_\bmu \mid \bmu \in \OP(n) \}$ given by the conjugacy classes in $B_n$ (see Section \ref{centralcharacters}). The following lemma is the final technical step needed to prove Theorems \ref{Theorem1} and \ref{Hurwitz}. Its proof will be left to Section \ref{proofofkeylemma}.
\begin{lemma}\label{keylemma}
\[
\delta_\bmu = 2^{(\ell(\bmu-n))/2} \Delta_\bmu.
\]
\end{lemma}

Assuming Lemma \ref{keylemma}, let us complete the proof of Theorem \ref{Hurwitz} (and thus also Theorem \ref{Theorem1} which is a special case). According to Proposition \ref{propositionsergeev2}, the change of basis matrix between $\{\Delta_\bmu \mid \bmu \in \OP(n)\}$ and $\{e_\bnu \mid \bnu \in \SP(n)\}$ is $f^\bnu _\bmu$ (indeed, the central characters of a group may be defined as the change of basis matrix between the orthogonal idempotents and the conjugacy classes).

\begin{lemma}
\[
t(\bnu) = \left(d(\bnu)/n!\right)^2
\]
\end{lemma}
\begin{proof}
The linear map $t: \cZ(\cY_n) \to \C$ takes the value $1/n!$ on $\delta_{1^n} = \Delta_{1^n}$ and zero on all other $\Delta_\bmu$. Thus to prove the lemma we should express the idempotents $e_\bnu$ in terms of the basis of conjugacy classes $\{\Delta_\bmu\}$. This is a straightforward exercise in the representation theory of finite groups, complicated only slightly by the fact that we are dealing with spin representations rather than linear representations (see Section \ref{centralcharacters} for details). The result follows by noting the interpretation of $d(\bnu)$ in terms of the dimension of the module $V^\bnu$ in Proposition \ref{propositionsergeev2}.
\end{proof}

The spin Hurwitz formulas of Theorem \ref{Hurwitz} now follow from Proposition \ref{inv}.

\subsection{Recursion formulas of Lee-Parker}\label{leeparker}
Let $\Sigma_g^{p}$ denote a closed spin surface of genus $g$ and parity $p\in \Z/2\Z$.  The following theorem was proved by Lee and Parker in \cite{LPP}. Let $C(\bmu)$ denote the centralizer of the conjugacy class $\bmu$ in $S_n$. Note that $\# C(\bmu) = \prod _i \mu_i (\mu_i !)$.
\begin{theorem}
Let $\bmu^1 , \ldots , \bmu^k \in \OP(n)$. 
\begin{enumerate}
\item If $g=g_1 + g_2$, and $p=p_1 +p_2$, then for $0 \leq k_0 \leq k$,
 \[
\cH(\Sigma_g , \bmu^1 , \ldots, \bmu^k) = \sum_{\bmu\in \OP(n)}  \# C(\bmu) \cH(\Sigma_{g_1}, \bmu^1, \bmu^{k_0}, \bmu)\cH(\Sigma_{g_2}, \bmu, \bmu^{k_0 +1} , \ldots, \bmu^k).
\]
\item If $g\geq 2$, or if $(g,p) = (1,0)$, then
\[
\cH_n(\Sigma^{p}_g, \bmu^1 , \ldots , \bmu^k) = \sum_{\bmu \in \OP(n)} \# C(\bmu) \cH_n(\Sigma^{p}_{g-1},\bmu, \bmu, \bmu^1 , \ldots , \bmu^k).
\]
\end{enumerate}
\end{theorem}
\begin{proof}
Considered as a linear map $\C \to \C$, the number $Z_n(\Sigma^p_g)$ factors as
\[
\C \xrightarrow{Z_n(\Sigma^{\ast p_1}_{g_1})} Z_n(S^1_{ap}) \xrightarrow{Z_n(^\ast\Sigma_{g_2}^{p_2})} \C,
\]
where $\Sigma^{\ast p_1}_{g_1}$ means $\Sigma^{p_1}_{g_1}$ punctured at one point, considered as a cobordism $\emptyset^1 \to S^1_{ap}$, and $^\ast \Sigma^{p_2}_{g_2}$ means $\Sigma_{g_2}^{p_2}$ punctured at one point, considered as a cobordism $S^1_{ap} \to \emptyset^1$. Let us expand the element 
\[
Z_n(\Sigma_{g_1}^{\ast p_1})(1) \in Z_n(S^1_{ap})
\]
in the $\delta_\bmu$ basis. As $\langle \Z(\Sigma_{g_1}^{\ast p_1})(1), \delta_\bmu \rangle = \cH(\Sigma_{g_1}^{p_1}, \bmu)$, the coefficient of $\delta_\bmu$ in this expansion is 
\[
\cH(\Sigma_{g_1}^{p_1}, \bmu)\langle \delta_\bmu , \delta_\bmu \rangle^{-1}.
\]
 The number $\langle \delta_\bmu , \delta_\bmu \rangle$ is a genus zero spin Hurwitz number with two ramification points; it is easy to compute that this is equal to $\#C(\bmu)^{-1}$. Thus
\begin{align*}
\cH_n(\Sigma) = Z_n(\Sigma_{g_2}^{\ast p_2}) \left(\sum_{\bmu \in \OP(n)} \cH_n(\Sigma_{g_1}^{p_1}, \bmu)\#C(\bmu) \delta_\bmu \right) \\
= \sum_{\bmu \in \OP(n)} \# C(\bmu) \cH_n(\Sigma_{g_1}^{p_1}, \bmu) \cH_n(\Sigma_{g_2}^{p_2}, \bmu).
\end{align*}
This proves part (1) in the case $k=0$. The ramified case is proved similarly. 
To prove part (2), note that we can factor $Z_n(\Sigma^p_g)$ as
\[
 \C \xrightarrow{Z_n(\Sigma_{g-1}^{\ast \ast p})} Z_n(S^1_{ap})\otimes Z_n(S^1_{ap}) \xrightarrow{\langle -,- \rangle} \C,
\]
where $\Sigma_{g-1}^{\ast \ast p}$ is $\Sigma_{g-1}^{p}$ with two punctures, considered as a cobordism $\emptyset^1 \to S^1_{ap} \sqcup S^1_{ap}$. Note that the inner product $\langle -,-\rangle$ the value of $Z_n$  on a 2-sphere with two punctures, considered as a cobordism $S^1 _{ap} \sqcup S^1_{ap} \to \emptyset^1$. The proof now proceeds in the same way as part (1).
\end{proof}

\subsection{Boundary conditions for extended TQFT}
To prove Lemma \ref{keylemma}, we need to compare the element $\delta_\bmu \in Z_n(S^1_{ap})$ with $\Delta_\bmu \in \cZ(\cY_n)$. To describe the embedding $Z_n(S^1_{ap})\hookrightarrow \cY_n$ explicitly, it is necessary to make use of a cobordism which is does not appear naturally in the category $\Bord^{Spin}$, namely, the whistle or closed-to-open string transition. We will briefly explain what this means below.

Let us first describe the category $\Bord^{Or,oc}$ of \emph{oriented} cobordisms with boundary condition (the theory of cobordism categories with boundary conditions is described in Section 4.3.22 of \cite{Lur}; we will need only a small part of the theory). In general, the 2-category $\Bord^{Or,oc}$ is defined in the same way as $\Bord^{Or}$, except that the word ``manifold'' is replaced by ``manifold with (marked) boundary''. Thus, $0$-objects of $\Bord^{Or,oc}$ are the same as in $\Bord^{Or}$, but there is an additional 1-object, given by the line interval $I=[0,1]$, where the boundary $\{0,1\}$ is considered ``marked'' (thus $I$ is a morphism $\emptyset^0 \to \emptyset^0$ rather than $pt_+ \to pt_+$, say).  If $M_0$ and $M_1$ are both 1-objects in $\Bord^{Or, oc}$ (thus, a disjoint union of copies of $S^1$ and $I$), then a $2$-morphism between $M_0$ and $M_1$ is a $2$-manifold with corners $\Sigma$, where $\partial \Sigma$ is equipped with a isomorphism
\[
\partial \Sigma \simeq \overline{M_0} \sqcup_{\partial_{marked} \overline{M_0}} \sqcup_{\partial_{marked} M_1} M_1. 
\]
\begin{remark}
The 1-category $\Omega \Bord^{Or,oc} = \Hom_{\Bord^{Or,oc}}(\emptyset^0,\emptyset^0)$ is equivalent to the category of \emph{open-closed cobordisms} as described e.g. in \cite{moore_d-branes_2006}.
\end{remark} 

Given an oriented 2d TQFT $Z: \Bord^{Or} \to \Alg$ with $Z(pt_+)=A$, an extension of $Z$ to a symmetric monoidal functor $\widetilde{Z}:\Bord^{Or,oc} \to \Alg$ is called a \emph{boundary condition} for $Z$. Boundary conditions for $Z$ are in one-to-one correspondence with $A$-modules; given an $A$-module $M$, there is an extension $\widetilde{Z}_M$ with $\widetilde{Z}_M(I) = \End(M)$. (More generally, we could consider categories of cobordisms where we allow marked boundary of various colours. An extension of $Z$ to such a category corresponds to a choice of $A$-module $M_i$ for each colour $i$. Such an extension assigns the vector space $\Hom_A(M_i,M_j)$ to the line interval in which one endpoint has colour $i$ and the other has colour $j$.)

Thus, any oriented TQFT $Z$ as above, has a canonical boundary condition $\widetilde{Z}$ given by the $A$-module $A$ itself. Thus $\widetilde{Z}(I) =A$. There are various $2$-morphisms given by marking intervals along the boundary of a disc. For example, marking 3 intervals along the boundary of a disc defines the \emph{pair of chaps} cobordism
\[
\CHAPS : I \sqcup I \to I.
\]
Applying $\widetilde{Z}$ to $\CHAPS$ recovers the multiplication map for the algebra $A$ (respectively, comultiplication if read in the other direction). Similarly, marking one interval on the boundary of the disc gives the trace map defining the Frobenius algebra structure on $A$ (or the unit map, if read in the other direction).

The main example that we will need is the \emph{whistle} cobordism (also known as the closed-to-open string transition),
\[
\WHISTLE : S^1 \to I
\]
The whistle has a cylinder $S^1 \times [0,1]$ as its underlying manifold, where the marked portion of the boundary is a line segment embedded in $S^1 \times \{1\}$. Applying $\widetilde{Z}$ to the whistle cobordism recovers the inclusion of the centre $\cZ(A) \hookrightarrow A$ (or the projection $A \to \Ab(A)$ if read in the other direction).

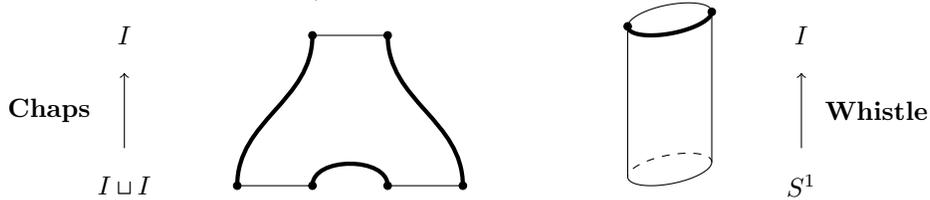
\begin{figure}[h]
\caption{The chaps and whistle cobordisms. The marked portion of the boundary is shown in bold.}
\begin{tikzpicture}

\draw (-.5,0) -- (.5,0);
\draw [fill] (-.5,0) circle [radius=0.05cm];
\draw [fill] (.5,0) circle [radius=0.05cm];
\draw [fill] (-1.5,-2) circle [radius=0.05cm];
\draw [fill] (1.5,-2) circle [radius=0.05cm];
\draw [fill] (-.5,-2) circle [radius=0.05cm];
\draw [fill] (.5,-2) circle [radius=0.05cm];

\draw (-1.5,-2) -- (-.5,-2);

\draw (.5,-2) -- (1.5,-2);

\draw [ultra thick] (-.5,0) to [out=270, in=90] (-1.5,-2);  

\draw [ultra thick] (.5,0) to [out=270, in=90] (1.5,-2); 

\draw [ultra thick] (-.5,-2) to [out=90, in=90] (.5,-2); 

\begin{scope}[xshift=-1cm]
\draw node at (-2,-2) {$I \sqcup I$};
\draw [->] (-2,-1.5) -- (-2,-.5);
\draw node at (-2,0) {$I$};
\draw node at (-3,-1) {$\CHAPS$};
\end{scope}

\begin{scope}[xshift=4cm, yshift=-2cm, scale=0.5]
\draw (1.618,.6266) -- (1.618,4.6266);

\draw (-0.618,.2394) -- (-.618,4.2394);

\draw [dashed, domain=20:200, smooth, variable=\t]
plot ({cos(\t) + 0.5*sin(\t)+.5},{0.433* sin(\t) +0.433});
\draw [domain=200:380, smooth, variable=\t]
plot ({cos(\t) + 0.5*sin(\t)+.5},{0.433* sin(\t) +0.433});

\draw [ domain=20:200, smooth, variable=\t]
plot ({cos(\t) + 0.5*sin(\t)+.5},{0.433* sin(\t) +4.433});
\draw [ultra thick, domain=200:380, smooth, variable=\t]
plot ({cos(\t) + 0.5*sin(\t)+.5},{0.433* sin(\t) +4.433});

\draw [fill] (1.618,4.6266) circle [radius=0.1cm];
\draw[fill] (-.618,4.2394) circle [radius=0.1cm];

\end{scope}

\begin{scope}[xshift=6cm]
\draw node at (0,-2) {$S^1$};
\draw [->] (0,-1.5) -- (0,-.5);
\draw node at (0,0) {$I$};
\draw node at (1.0,-1) {$\WHISTLE$};
\end{scope}

\end{tikzpicture}
\end{figure}

The theory of boundary conditions for TQFTs carries over to the spin setting: there is a 2-category $\Bord^{Spin, oc}$, with the property that extensions $\widetilde{Z}$ of a TQFT $Z:\Bord^{Spin} \to S\Alg$ correspond to supermodules for the superalgebra $A=Z(pt_+)$. Naturally, there are some new features in the spin case: there are now two $1$-objects $I$ and $\overline{I}$ whose underlying manifold is the interval $[0,1]$. The canonical boundary condition $\widetilde{Z}$ for $Z$ assigns $\widetilde{Z}(I)=A$ and $\widetilde{Z}(\overline{I})=A^!$, and the various spin structures on the pair of chaps cobordisms recover the multiplication on $A$, the $A$-module structure on $A^!$, and the map $A^! \otimes A^! \to A$ expressing the triviality of the square of the Serre automorphism. We will only require the use of a single cobordism in $\Bord^{Spin,oc}$, namely the whistle cobordism $S^1_{ap} \to I$. 

\subsection{The map $\widetilde{Z}_n(\WHISTLE)$}
Consider the parity theory $Z_1$ with its canonical boundary condition $\widetilde{Z}_1$ and let $A$ denote $\Cl_1 = Z_1(pt_+)$. The construction of the covering theories $Z_n$ as in Section \ref{average} can be extended to cobordisms with boundary conditions. For example, to compute $\widetilde{Z}_n(I)$, we look at the groupoid $\Cov_n(I)$. Note that the boundary of $I$ is marked, and thus we should consider $n$-fold covers of $I$ equipped with trivializations at $\partial I = \{0,1\}$. Thus $\Cov_n(I)$ is isomorphic to the discrete groupoid $S_n$, and $\widetilde{Z}_n(I) = \bigoplus_{S_n} A^{\otimes n}$. Applying $\widetilde{Z}_n$ to the pair of chaps cobordism identifies the algebra structure on $\bigoplus _{S_n} A^{\otimes n}$ as the semidirect product $A^{\otimes n} \rtimes S_n$.

Now let us apply the covering theory construction of Section \ref{average} to the whistle cobordism. Consider the following correspondence of groupoids:
\begin{equation}\label{equationcorrespondence}
\Cov_n(S^1_{ap}) \xleftarrow{} \Cov_n(\WHISTLE) \xrightarrow{} \Cov_n(I).
\end{equation}

Objects of the groupoid $\Cov_n(S^1_{ap})$ are given by covers of the circle $S^1_{ap}$. Equivalence classes of such covers are indexed by $\PP(n)$. For each $\bmu \in \PP(n)$, we have a cover $\cS_\bmu$ which is a disjoint union of circles $\cS_{\mu_i}$ each of which are connected covers of $S^1_{ap}$ of degree $\mu_i$. 

Objects of $\Cov_n(\WHISTLE)$ are given by covers of $(S^1_{ap}, p)$ where $p$ is a choice of basepoint on $S^1_{ap}$, together with a framing of the cover at $p$ (i.e. an identification between the fibre over $p$ and the set $\{1,\ldots, n\}$). Such covers are indexed by $S_n$. For each $\sigma \in S_n$, we have such a cover $\cS_\sigma$ which is a disjoint union of circles $\cS_{\sigma_i}$ where $\sigma = \sigma_1 \ldots \sigma_\ell$ is a disjoint cycle decomposition. 

Finally, objects of $\Cov_n(I)$ are covers of $(I=[0,1], \{0,1\})$ with framings at $\{0,1\}$. Such covers are indexed by $S_n$. For each $\sigma \in S_n$, we have $I_\sigma$ which is a disjoint union of $n$ line intervals.

Following the construction of Section \ref{average}, $\widetilde{Z_n}$ assigns to the whistle cobordism a linear map, obtained by averaging the theory $Z_1$ over the correspondence of groupoids \ref{equationcorrespondence}:
\begin{align}\label{equationlinearmap} \widetilde{Z}_n(S^1_{ap}) &= \bigoplus_{\bmu \in \PP(n)} Z_1(\cS_\bmu)^{C(\bmu)} \xrightarrow{\widetilde{Z}_n(\WHISTLE)} \bigoplus _{\sigma \in S_n} Z_1(I_\sigma) = Z_n(S^1_{ap}) 
\end{align}

\subsection{Proof of Lemma \ref{keylemma}}\label{proofofkeylemma}
Let us describe $\widetilde{Z}_n(\WHISTLE)$ more explicitly. To begin with we have the natural inclusion:
\[
\bigoplus_{\bmu \in \PP(n)} Z_1(\cS_{\bmu})^{C(\bmu)} \simeq \left( \bigoplus _{\sigma \in S_n} Z_1(\cS_\sigma)\right)^{S_n} \hookrightarrow \bigoplus _{\sigma \in S_n} Z_1 (\cS_\sigma).
\]
As explained in Section \ref{proofoftft}, the action of $C(\bmu)$ on the one dimensional super vector space $\Z_1(\cS_\bmu)$ is trivial if $\bmu \in \OP(n)$ and non-trivial otherwise. Thus we need only consider $\bmu \in \OP(n)$.

Fix an element $\sigma \in S_n$ of cycle type $\bmu \in \OP(n)$, with a corresponding disjoint cycle decomposition $\sigma = \sigma_1 \ldots \sigma_\ell$. We have $Z_1(\cS_\sigma) = Z_1(\cS_{\sigma_1}) \otimes \ldots \otimes Z_1 (\cS_{\sigma_\ell})$. Each $\cS_{\sigma_i}$ is a $\mu_i$-fold cover of the base $S^1_{ap}$, thus we compute $Z_1(\cS_{\mu_i})$ as a tensor product $\mu_i$ copies of $A^!$ over $A$, arranged in a circle (see Figure \ref{figurecircles}). We can write this circular tensor product linearly as: 
\begin{equation}\label{equationcircular}
Z_1(\cS_{\sigma_i}) = \underbrace{(A^! \otimes _A \ldots \otimes_A A^!)}_\text{$\mu_i$} \otimes_{A^e} A.
\end{equation}
To describe the map $Z_1(\cS_{\sigma_i}) \to Z_1(I_{\sigma_i})$, note that each factor $A^!$ canonically embeds in $A^e$. Replacing each instance of $A^!$ by $A^e$ in the tensor product \ref{equationcircular}, we arrive at $A^{\otimes \mu_i}$ which is precisely $Z_1(I_{\sigma_i})$.

\begin{figure}[h]
\caption{Computing $\widetilde{Z}_n(\WHISTLE)$.}\label{figurecircles}
\begin{tikzpicture}
\begin{scope}[scale= 0.7]
\draw node at (2,0) {$\otimes_A$};
\draw node at (1,1.732) {$\otimes_A$};
\draw node at (-2,0) {$\otimes_A$};
\draw node at (-1,1.732) {$\otimes_A$};
\draw node at (-1,-1.732) {$\otimes_A$};
\draw node at (1,-1.732) {$\otimes_A$};

\draw node at (0,2) {$\dotsb$};
\draw node at (1.732,1) {$A^!$};
\draw node at (0,-2) {$A^!$};
\draw node at (-1.732,1) {$A^!$};
\draw node at (1.732,-1) {$A^!$};
\draw node at (-1.732,-1) {$A^!$};
\end{scope}

\begin{scope}[scale= 0.7, xshift=7cm]
\draw node at (2,0) {$\otimes_A$};
\draw node at (1,1.732) {$\otimes_A$};
\draw node at (-2,0) {$\otimes_A$};
\draw node at (-1,1.732) {$\otimes_A$};
\draw node at (-1,-1.732) {$\otimes_A$};
\draw node at (1,-1.732) {$\otimes_A$};

\draw node at (0,2) {$\dotsb$};
\draw node at (1.732,1) {$A^e$};
\draw node at (0,-2) {$A^e$};
\draw node at (-1.732,1) {$A^e$};
\draw node at (1.732,-1) {$A^e$};
\draw node at (-1.732,-1) {$A^e$};
\end{scope}

\begin{scope}[xshift = 9cm]
\draw node at (0,0) {$A \otimes \ldots \otimes A$};
\end{scope}

\draw node at (2.5,0) {$\hookrightarrow$};

\draw node at (7.5,0) {$\simeq$};

\begin{scope}[scale=0.7, yshift=-6cm]
\draw ({2*cos(130)},{2*sin(130)}) arc [radius=2cm, start angle =130, end angle=410];
\draw [fill] ({2*cos(30)},{2*sin(30)}) circle [radius=0.1cm];
\draw [fill]  ({2*cos(150)},{2*sin(150)}) circle [radius=0.1cm];
\draw [fill]  ({2*cos(210)},{2*sin(210)}) circle [radius=0.1cm];
\draw [fill]   ({2*cos(270)},{2*sin(270)}) circle [radius=0.1cm];
\draw [fill]  ({2*cos(330)},{2*sin(330)}) circle [radius=0.1cm];

\draw node at (0,2) {$\ldots$};

\draw node at (0,0) {$S_{\sigma_i}$};

\draw node at (0,-3) {$\downarrow$};

\end{scope}

\begin{scope}[scale=0.7, xshift=7cm, yshift=-6cm]
\draw ({2*cos(160)}, {2*sin(160)}) arc [radius=2cm, start angle=160, end angle=200];
\draw [fill] ({2*cos(160)}, {2*sin(160)}) circle [radius=0.1cm];
\draw [fill] ({2*cos(200)}, {2*sin(200)}) circle [radius=0.1cm];

\draw ({2*cos(220)}, {2*sin(220)}) arc [radius=2cm, start angle=220, end angle=260];
\draw [fill] ({2*cos(220)}, {2*sin(220)}) circle [radius=0.1cm];
\draw [fill] ({2*cos(260)}, {2*sin(260)}) circle [radius=0.1cm];

\draw ({2*cos(280)}, {2*sin(280)}) arc [radius=2cm, start angle=280, end angle=320];
\draw [fill] ({2*cos(280)}, {2*sin(280)}) circle [radius=0.1cm];
\draw [fill] ({2*cos(320)}, {2*sin(320)}) circle [radius=0.1cm];

\draw ({2*cos(340)}, {2*sin(340)}) arc [radius=2cm, start angle=340, end angle=380];
\draw [fill] ({2*cos(340)}, {2*sin(340)}) circle [radius=0.1cm];
\draw [fill] ({2*cos(380)}, {2*sin(380)}) circle [radius=0.1cm];

\draw node at (0,2) {$\ldots$};

\draw node at (0,0) {$I_{\sigma_i}$};

\draw node at (0,-3) {$\downarrow$};

\end{scope}

\begin{scope}[scale=0.7, xshift=12cm, yshift=-6cm]
\draw (0,.7) -- (0,-.7);
\draw [fill] (0,.7) circle [radius=0.1];
\draw [fill] (0,-.7) circle [radius=0.1];

\draw node at (1,0) {$\ldots$};

\draw (2,.7) -- (2,-.7);
\draw [fill] (2,.7) circle [radius=0.1];
\draw [fill] (2,-.7) circle [radius=0.1];

\end{scope}

\begin{scope}[scale=0.7, yshift=-11cm, ]
\draw (0,0) circle [radius=1cm];
\draw [fill] (0,-1) circle [radius=0.1cm];
\end{scope}

\begin{scope}[scale=0.7, yshift =-11cm,  xshift=7cm]
\draw ({cos(-60)},{sin(-60)}) arc [radius=1cm, start angle= -60, end angle=240];

\draw [fill] ({cos(-60)},{sin(-60)}) circle [radius=0.1cm];
\draw [fill] ({cos(240)},{sin(240)}) circle [radius=0.1cm];
\end{scope}

\draw node at (2.5,-4.3) {$\leftarrow$};

\draw node at (7.5,-4.3) {$\simeq$};

\end{tikzpicture}
\end{figure}

On the other hand, $\cS_{\sigma_i}$ is isomorphic to $S^1_{ap}$, and thus $Z_1(\cS_\bmu) \simeq \C$. To go between the two descriptions of $Z_1(\cS_{\sigma_i})$, we can use the trivialization of the Serre automorphism $A^! \otimes _A A^! \simeq A$ to contract pairs of $A^!$'s in the tensor product, eventually arriving at $A^! \otimes _{A^e} A \simeq \C $. 

\begin{lemma}\label{subkeylemma}
Under the map $\C \simeq Z_1(S_{\sigma_i}) \to Z_1(I_{\sigma_i}) \simeq A^{\otimes \mu_i}$, the element $1$ maps to
\[
\sum 2^{-(\mu_i -1)/2} \eta ^{c_1} \otimes \ldots \otimes \eta^{c_{\mu_i}}
\]
where the sum is over $(c_1 , \ldots , c_{\mu_i}) \in (\Z/2\Z)^{\mu_i}$ such that $\sum_j c_j = 0 \in \Z/2\Z$.
\end{lemma}
\begin{proof}
Note that $A^!$ can be identified with the sub-bimodule of $A^e$ generated by the element $\varepsilon = 1\otimes 1 + \eta \otimes \eta$. Unwinding the duality data for $A^!$, we obtain that under the trivialization of the square of the Serre automorphism, the element $1\in A$ corresponds to the tensor $(1/2) \varepsilon \otimes \varepsilon$ in $A^! \otimes _A A^!$. Also, the element $1\in \C$ corresponds to $(1/2) \varepsilon \otimes 1$ in $A^! \otimes_{A^e} A$.

It follows that the element $1 \in \C$ maps to $(2^{-(\mu_i + 1)/2} \epsilon \otimes \ldots \otimes \epsilon \otimes$ in the tensor product \ref{equationcircular}.  Lemma \ref{subkeylemma} follows by taking canonical embedding $A^! \hookrightarrow A^e$ for each $A^!$ in the tensor product, and simplifying to $A^{\otimes \mu_i}$.
\end{proof}

Now let us complete the proof of Lemma \ref{keylemma}. Consider the tensor product over $i=1, \ldots, \ell$ of the maps $Z_1(\cS_{\sigma_i}) \to Z_1(I_{\sigma_i})$. By Lemma \ref{subkeylemma},the element $\delta_\bmu \in Z_1(\cS_\sigma)$ maps to 
\[
\bigotimes_{i=1} ^\ell \left(\sum 2^{-(\mu_i -1)/2} \eta ^{c_1} \otimes \ldots \otimes \eta^{c_{\mu_i}} \right) = 2^{(\ell -n)/2} \Delta_\sigma \in A^{\otimes n},
\]
where the sum is over $(c_1 , \ldots , c_{\mu_i}) \in (\Z/2\Z)^{\mu_i}$ such that $\sum_j c_j = 0 + 2\Z$. Here $\bigotimes_{i=1} ^\ell A^{\mu_i} = A^{\otimes n}$ is identified with the subspace $A^{\otimes n} \otimes \sigma$ of $\cY_n$, and $\Delta_\sigma$ is defined as in Proposition \ref{propositiondelta}. Thus $\widetilde{Z}_n(\WHISTLE)$ maps $\delta_\bmu$ to $2^{(\ell -n)/2} \Delta_\bmu$ as required.

\section{The Parity TQFT}\label{AtiyahTheory}
In this section, we explain how the parity of a closed spin surface $\Sigma$ can be upgraded to a TQFT
\[
Z_p:\Bord^{Spin} \to S\Alg,
\]
such that $Z_p(\Sigma) = (-1)^{p(\Sigma)}$. We will give a direct, but somewhat ad-hoc construction, then in Subsection \ref{homotopical} we explain how this theory can be constructed via homotopical methods using the spin orientation of $KO$.
\begin{remark}
The TQFT $Z_p$ is \emph{invertible} in the sense that for each object $N$, $Z_p(N)$ is (weakly) invertible for the monoidal structure, and for each $1$ or $2$-morphism $M$, $Z_p(M)$ is a (weakly) invertible $1$ or $2$-morphism. The parity TQFT and more general invertible TQFTs are discussed in \cite{Freed2}.
\end{remark}

\subsection{An aside on conformal spin structures and metrics}\label{Spin}
For the construction of the TQFT, it will be helpful to make use of metrics on the various vector bundles associated to the spin structure. Rather than making choices and then proving the invariance of those choices it will be more convenient to work with a particular model of the Spin bordism category for which canonical metrics are built in. The definitions and conventions in this subsection are drawn from \cite{ST}.

Let $V$ be a finite dimensional (real or complex) vector space with a quadratic form $q$. The Clifford algebra $\Cl(V)$ is the quotient of the tensor algebra on $V$ by the relation $v^2 = - q(v) .1$. We write $\Cl_d (\R)$ for the Clifford algebra of $\R ^d$ with its canonical positive definite quadratic form and $\Cl_{-d}( \R)$ for the Clifford algebra of $\R ^d$ with its negative definite form. We can extend the involution $v \mapsto -v$ on $V$ uniquely to an algebra involution $\alpha$ on $\Cl(V)$. This gives it the structure of a superalgebra (i.e. $\Z/2\Z$-graded algebra), and we can speak of a supermodules, etc. Note that the Clifford algebra $\Cl(V)$ inherits an inner product from that on $V$.

A \emph{spin structure} on $V$ is a choice of an irreducible $\Cl(V)-\Cl_d(\R)$ superbimodule $S_V$ with a compatible inner product (i.e. multipliction by elements of unit length in $V$ and $\R^n$ are isometries). Given two spin vector spaces $V$ and $W$, the space $Spin(V,W)$ of spin isometries consists of an isometry $f:V \to W$, together with a homomorphism of bimodules $\phi : f^\ast S_W \to S_V$ which is additionally an isometry. The map $(f,\phi) \mapsto f$ exhibits $Spin(V,W)$ as a double cover of $SO(V,W)$ (the connected component of $O(V,W)$ for which \emph{there exists} such a $\phi$). 

If $E \to X$ is a real vector bundle with a positive definite metric, let $\Cl(E) \to X$ denote the bundle of algebras whose fibre over $x \in X$ is the Clifford algebra $\Cl(E_x)$. A spin structure on $E$ is a bundle $S_E$ of irreducible $\Cl(E)-\Cl_d(\R)$ superbimodules (i.e. a compatible family of spin structure on the quadratic vector spaces $E_x$). 

Given a smooth manifold $M$ of dimension $d$, we write $L^k$ for the (trivializable) line bundle of $k$-densities on $M$ (in our convention a $d$-density can be integrated over $M$ to get a real number). The \emph{weightless cotangent bundle} $T^\ast_0 M$ is defined to be $L^{-1} \otimes T^\ast M$. A \emph{conformal spin structure} on a conformal manifold $M$ is defined to be a spin structure on $T^\ast_0 M$.

Note that a choice of conformal structure equips the weightless cotangent bundle with a \emph{canonical} metric, and thus the Clifford bundles, spinor bundles, etc. all come equipped with canonical metrics. Moreover, if $M$ has boundary $\partial M$, then $T^\ast _0M |_{\partial M}$ is canonically equivalent to $T^\ast_0 \partial M \oplus \underline{\R}$, thus a conformal spin structure on $M$ canonically defines a conformal spin structure on the boundary. 

We can define a $2$-category of (topological) conformal spin cobordisms $c\Bord^{Spin}$ in a similar way to Subsection \ref{extended}. Note that $0$ and $1$-dimensional manifolds carry a unique conformal structure. Moreover, we will consider two conformal spin $2$-cobordisms to give rise to the same $2$-morphism if there is a spin diffeomorphism between them, fixing the given isomorphisms on the boundary (such a diffeomorphism is not required to preserve the conformal structure). 

One can check that the symmetric monoidal $2$-category $c\Bord^{Spin}$ is equivalent to $\Bord^{Spin}$ as defined, for example in \cite{SP}. We will not attempt to give a rigorous proof of this statement\footnote{Indeed to give such a rigorous proof, we would first need to give a more rigorous definition of the bordism bicategory, taking care of gluing using collars e.g. as in \cite{SP}.}, but intuitively this should be clear: a choice of Riemannian metric on a conformal spin manifold trivializes the density bundle and defines a spin structure in the usual sense. Moreover, the space of such Riemannian metrics is contractible.

\subsection{The direct construction}
The construction of the $0,1$-dimensional part of the theory is taken from \cite{ST}, Section 2.3. 
Given a (conformal) spin $0$-manifold $N$, $\cS(N)$ is a finite dimensional super vector space equipped with an inner product. We define a symmetric bilinear form $b_N$ by $b_N(v,w) = \langle v, i(w) \rangle$, where $i$ is the grading involution and $\langle -,-\rangle$ is the inner product. We set
\[
Z_p(N) = \Cl(\cS(N), b_N).
\]

Given a spin 1-cobordism $M$ with $\partial N = \overline{N}_0 \sqcup N_1$, the even spinor bundle $S_M$ on $M$ is a metric line bundle and thus is equipped with a canonical flat connection. The space of flat sections $\cS^h(M)$ of $S_M$ over $M$ is called the space of harmonic spinors. The space $\cS_0^h(M)$ of even harmonic spinors is finite dimensional, and comes equipped with a map 
\[
L:\cS^h_0(M) \to \cS(\partial N),
\]
given by restriction to the boundary. The morphism $L$ is a \emph{generalized Lagrangian} in the sense of \cite{ST}, Definition 2.2.9. As explained in \cite{ST}, Definition 2.2.4, the \emph{Fock module} of $L$ is defined to be
\[
F(L) := \ker(L)^\ast \otimes \Lambda^\ast (Im(L)),
\]
which carries a canonical action of $\Cl(\cS(\partial N))$, or equivalently, is a $\Cl(\cS(N_0))-\Cl(\cS(N_1))$-bimodule. We set $Z_p(M) = F(L)$. Note that if $M$ is closed, then $F(L) = \bigwedge^{top}(\cS^h(M))$ is the \emph{Pfaffian line} of the Dirac operator on $T^\ast _0 M$. For example, if $M \simeq S^1_{ap}$, there are no harmonic spinors and thus $Z_p(S^1_{ap})$ is canonically isometric to $\R$. On the other hand, if $M \simeq S^1_{per}$ the space of harmonic spinors is one dimensional and thus $Z_p(M)$ is a one dimensional odd metric line. 

\begin{lemma}[\cite{ST}, Gluing Lemma 2.2.8]\footnote{Note that in this case, all Hilbert spaces and Clifford algebras are finite dimensional.}
Given conformal spin $1$-cobordisms $N_0 \xrightarrow{M} N_1$ and $N_1 \xrightarrow{M'} N_2$ there is a canonical isomorphism $Z_p(M \sqcup_{N_1} M') \simeq Z_p(M) \otimes_{Z_p(N_1)} Z_p(M')$.
\end{lemma}

It remains to explain what $Z_p$ assigns to a $2$-morphism in $\Bord^{Spin}$. Let us first consider the case of a spin surface $\Sigma$ with boundary $M:= \partial \Sigma$. We would like to define an element $Z_p(\Sigma) \in Z_p(M)$. There are various approaches to this; for example one could use a refinement of the mod 2 index of the Dirac operator to a spin conformal $2$-manifold with boundary which takes values in the Pfaffian line of the boundary. However, we will just pursue a more down to earth method in terms of the parity of closed spin manifolds obtained by gluing in discs and cylinders to $\Sigma$.

Let $D$ be the unit disc in $\R^2$ equipped with its canonical conformal spin structure, and $C = S^1_{per} \times [0,1]$ similarly equipped. 
\begin{definition}
Suppose that $M$ is a closed conformal spin $1$-manifold with $n_D$ antiperiodic components and $2n_C$ antiperiodic components (where $n_D, n_C \in \Z_{\geq 0}$). A \emph{capping} of $M$ is defined to be a spin isomorphism 
\[
c: M \simeq \partial (\bigsqcup_{n=1}^{n_D} D \sqcup \bigsqcup_{n=1}^{n_C} C).
\]
The set of cappings of $M$ is denoted $Cappings(M)$.
\end{definition}
Note that a capping of $M$ only exists if $M$ has an even number of periodic components. This condition is satisfied if $M$ is the boundary of a spin surface.

\begin{lemma}\label{lemmacapping}
Each capping of $M$ canonically determines an element of $Z_p(M)$ of unit length.
\end{lemma}
\begin{proof}
First consider the case where $M$ has only anti-periodic components; the space of harmonic spinors is $0$, and thus its determinant line $Z_p(M)$ is canonically trivial. In particular $Z_p(M)$ carries the canonical element $1$. Now let us consider the case when $M$ has precisely two periodic components $M_0$ and $M_1$. A choice of capping allows us to define a parallel transport isometry $\phi: \cS^h(\overline{M}_0)  \simeq \cS^h(M_1)$ which in turn defines an element $\det(\phi) \in Z_p(M)$. The general case follows.
\end{proof}

Given a spin surface $\Sigma$ with boundary $\partial \Sigma = M$, a capping  $c$ of $M$ canonically determines a closed spin surface $\Sigma_c$. by gluing along the map $c$. 
\begin{lemma}
The assignment 
\begin{align*}
Cappings(M) &\to \{\pm 1 \} \subseteq \R \\
c &\mapsto (-1)^{p(\Sigma_c)} .
\end{align*}
factors through the map $Cappings(M) \to Z_p(M)$ given by Lemma \ref{lemmacapping}, defining an isometry of graded metric lines
\[
Z_p(\Sigma): Z_p(M) \to \R.
\]
\end{lemma}
\begin{proof}
We prove this using the definition of parity via the Arf invariant of the quadratic form $\varphi: H_1(\Sigma_c ; \F_2) \to \F_2$ associated to the spin structure. Recall that if $M$ is a closed loop in $\Sigma$, then $\varphi([M]) = 0$ (respectively, $1$) if $M$ has a antiperiodic (respectively, periodic) spin structure. 

Changing the choice of capping on antiperiodic components does not affect the corresponding quadratic form, as one can choose representatives of cycles in $H_1(\Sigma_c ; \F_2)$ which do not intersect the antiperiodic components. Thus we may assume $M$ has only periodic components. Consider a particular component $M_0$ of $M$; recall that there are two choices of spin isomorphism $S^1_{per} \simeq M_0$ (up to spin isotopy). Given a capping $c$, we can modify $c$ by changing the choice of $S^1_{per} \simeq M_0$ to obtain a new capping $c'$. The effect of changing $c$ to $c'$ has presicely the effect of changing the value of the quadratic form on the Poincar\'e dual cycle to $M_0$, which in turn changes the value of the Arf invariant; simultaniously, the corresponding value of $Z_p(M)$ gets multiplied by $-1$. The result follows readily from this observation.
\end{proof}

Finally, given an arbitrary $2$-morphism $\Sigma: M_0 \to M_1$, where $M_0,M_1: N_0 \to N_1$ are $1$-cobordisms, we can now define a map
\[
Z_p(\Sigma): Z_p(M_0) \to Z_p(M_1),
\]
of $Z_p(N_0)-Z_p(N_1)$-bimodules as follows. Note that, topologically, the manifold with corners underlying $\Sigma$ may be thought of as a spin manifold $\Sigma'$ with closed boundary $M'$ (i.e. we can smooth the corners of $\Sigma$). To identify the spin structure on the boundary $M'$, let us consider the underlying spin $1$-manifold with boundary of $M_0$, as a cobordism
\[
M_0' :N_0 \sqcup \overline{N_1} \to \emptyset^0.
\]
Taking the opposite spin structure, we have 
\[
\overline{M}_0 ': \overline{N_0} \sqcup N_1 \to \emptyset^0.
\]
Similarly, we can consider
\[
M_1': \emptyset^0 \to \overline{N_0} \sqcup N_1.
\]
Thus, we can compose $\overline{M}_0'$ and $M_1 '$ in the bordism category to obtain the closed spin $1$-manifold $M'$. The general principles of preservations of duality under symmetric monoidal functors tell us that:
\[
Z_p(M') \simeq \Hom_{Z_p(N_0)^{op} \otimes Z_p(N_1)} (Z_p(\overline{M}_0 '), Z_p(M_1 ')).
\]
We define $Z_p(\Sigma) \in \Hom_{Z(N_0)^{op} \otimes Z(N_1)}(Z(M_0), Z(M_1))$ to correspond to $Z_p(\Sigma') \in \Z_p(M')$ under the above isomorphism. 

It remains to check that $Z_p$ as defined above, defines a TQFT, i.e. a symmetric monoidal functor
\[
Z_p: \Bord^{Spin} \to S\Alg_{\R, m}^\times,
\]
where $S\Alg_{\R,m}^\times$ is the $2$-category whose objects are Clifford algebras of quadratic vector spaces, morphisms are Fock modules associated to Lagrangian subspaces, and $2$-morphisms are isomorphisms of bimodules which are also isometries for the given inner products (see \cite{ST}, Definition 2.2.4). (We obtain a TQFT with target $S\Alg$ by just applying $\otimes_ \R \C$ and forgetting the inner products.) This amounts to checking the compatibility of the above definition under various kinds of gluing. We will not spell out all the necessary compatibilities here; rather, we will prove the following gluing law from which all necessary compatibilities may be deduced.
\begin{proposition}\label{gluing}
Let $\Sigma$ be a closed spin surface, and $M$ an embedded disjoint union of circles in $\Sigma$ with an even number of periodic components. Define $\Sigma_{cut, M}$ by cutting along $M$ to obtain a surface with boundary $M \sqcup \overline{M}$. Choose a capping $c$ of $M$, and let $\overline{c}$ denote the corresponding capping of $\overline{M}$. We define the closed surface $\Sigma_{c \sqcup \overline{c}, M}$ by capping the boundary of $\Sigma_{cut, M}$ according to $c \sqcup \overline{c}$. Then,
\[
p(\Sigma) = p(\Sigma_{c \sqcup \overline{c}, M}).
\]
\end{proposition}
\begin{proof}
We use the definition of parity via the Arf invariant of the quadratic form on $H_1$. The proof is an elementary check based on the properties of embedded spin circles under cutting and gluing. We will give the argument in the case where $M$ has two periodic components (the case when $M$ has a single antiperiodic component is proved similarly, and all other cases can be proved by indection from these). For simplicity, we assume further that $\Sigma$ is connected.

Note that the genus of $\Sigma_{c \sqcup \overline{c}, M}$ is always one greater than the genus of $\Sigma$. One can choose a symplectic basis for $H_1(\Sigma ; \F_2)$ which contains the class $\kappa$ of either component of $M$ (the two components necessarily define the same class in $H_1(\Sigma, \F_2)$) and its Poncar\'e dual $\kappa^\vee$. The $\F_2$-vector space $H_1(\Sigma_{c \sqcup \overline{c}} ; \F_2)$ has a similar basis but with two copies of $\kappa$, $\kappa_1$ and $\kappa_2$ and their respective Poncar\'e duals. Here, there are two possible cases: either $\kappa^\vee$ is periodic or anti-periodic. In the first case, precisely one of $\kappa_1^\vee$ or $\kappa_2^\vee$ will be periodic (which one depends on the choice of $c$). In the second case, either both $\kappa_1^\vee$ and $\kappa_2^\vee$ will be periodic, or both anti-periodic. Thus in all cases, the Arf invariant is unchanged as required. 
\end{proof}
\begin{remark}
One could also prove proposition \ref{gluing} by noting that the surfaces $\Sigma$ and $\Sigma_{c \sqcup \overline{c}, M}$ are spin cobordant, and the parity is a spin cobordism invariant \cite{At}.
\end{remark}

\subsection{Homotopical construction}\label{homotopical}
One way of thinking of the parity of a closed spin surface is in terms of the Gysin map in $KO$-theory afforded by the Spin structure. Namely, we have a map
\[
\int_\Sigma :\KO(\Sigma) \to \KO^{-2}(pt) \simeq \Z/2\Z,
\]
and $\int_\Sigma 1 = p(\Sigma)$. 

Similarly, one has invariants $\int_M 1 \in \KO^{-1}(pt) \simeq \Z/2\Z$, and $\int_N 1 \in \KO(pt) = \Z$ given a spin $1$-manifold $M$ and a spin $0$-manifold. There is a very general procedure by which one can upgrade the assignment $\Sigma \mapsto \int_\Sigma 1$ to a TQFT. 

First note that the Spin orientation of $KO$ may be thought of as a map of $E_\infty$-ring spectra 
\[
ABS: MSpin \to ko.
\]
Here, $MSpin$ is the spectrum representing spin cobordism, and $ko$ represents connective $KO$. Thus a closed spin $d$-manifold $M$ represents a class $[M] \in \Omega_d^{Spin} = \pi_d(MSpin)$ and the element $\int_M 1$ is equal to the image of $[M]$ in $\pi_d(KO) = KO^{-d}(pt)$ under the map $ABS$. Given a connective spectrum $E$ (or equivalently and infinite loop space), its fundamental $2$-groupoid $\pi_{\leq 2}(E)$ is naturally symmetric monoidal. The morphism ABS induces a functor of symmetric monoidal $2$-groupoids on the level of fundamental $2$-groupoids.
\[
\pi_{\leq 2} (MSpin) \to \pi_{\leq 2}(ko).
\]
It remains to describe the $2$-groupoids $\pi_{\leq 2} (MSpin)$  and $\pi_{\leq 2}(ko)$ in terms of bordism categories and superalgebras respectively. 
 
Let us define a 2-category $\Bord_\infty^{spin}$ as follows: the objects and morphisms of $\Bord_\infty^{spin}$ are the same as those of $\Bord^{Spin}$, but the 2-morphisms are given by 2-cobordisms \emph{up to cobordism} (recall that in the 2-category $\Bord^{Spin}$, the 2-morphisms are 2-cobordisms up to diffeomorphism). Note that there is a canonical quotient functor $\Bord^{Spin} \to \Bord_\infty ^{Spin}$. The following proposition is  surely known, but I have been unable to track down a precise reference; it should follow from an appropriate version of the results of \cite{GMTW}.
\begin{proposition}
The fundamental 2-groupoid of $MSpin$ is equivalent as symmetric monoidal 2-categories to $\Bord_{\infty}^{Spin}$.
\end{proposition}

On the other hand, we have the groupoid $S\Alg_{\R, m}^\times$ of \emph{invertible} real metric superalgebras, invertible metric superbimodules, and invertible isometries of bimodules. Again, the following result is surely known; we give a proof below.
\begin{lemma}\label{lemmacliff}
There is a canonical symmetric monoidal functor
\[
Cliff:\pi_{\leq 2} ko \to S\Alg_{\R, m}^\times.
\]
\end{lemma}
\begin{remark}
The functor $Cliff$ is not quite an equivalence: it induces an isomorphism on $\pi_1$ and $\pi_2$, but on $\pi_0$ it induces the quotient map $\Z \to \Z/8\Z$.
\end{remark}

Finally, the TQFT $Z_p$ is given by the composite:
\[
\Bord^{Spin} \to \Bord_{\infty}^{Spin}\simeq \pi_{\leq 2} MSpin \xrightarrow{\pi_{\leq 2} ABS} \pi_{\leq 2} (ko) \xrightarrow{Cliff} S\Alg_{\R, m}.
\]
\begin{proof}[Proof of Lemma \ref{lemmacliff}]
Let $\Vect _{\R,m}$ be the \emph{topological} category of real vector spaces equipped with an inner product, with isometries as morphisms. Recall that the space $ko$ is obtained by taking the group completion of this category. We will define a symmetric monoidal functor from $\Vect _{\R,m} $ (with symmetric monoidal structure given by orthogonal direct sum) to the 2-category $S\Alg_{\R, m}$, which by the universal property of group completion must factor through $ko$. To define such a functor, we must assign a superalgebra to each quadratic vector space $V$, a bimodule to each isometry $f: V\to W$, and a map of bimodules to each homotopy class of paths of isometries $f_t : V\to W$.

The map $Cliff$ is defined by taking a vector space $V$ as above to its Clifford algebra $\Cl(V)$. An isometry $f: V \to W$ naturally defines an invertible $\Cl(V)-\Cl(W)$ bimodule $M_f$. Recall that a \emph{spin structure} on $W$ means a choice of an irreducible $\Cl (W) - \Cl_d(\R)$ bimodule. Note that for any choice of spin structure $S_W$ on $W$, $M_f$ can be identified with $f^\ast S_W \otimes _{\Cl _d(\R)} S_W ^\ast$, where $S_W ^\ast = \Hom _\R (S_W , \R)$ is the $\Cl _d(\R)- \Cl(W)$-bimodule dual to $S_W$. Moreover, picking a compatible spin structure $S_V$ on $V$ and a lift of $f$ to a spin isometry $(f,\phi)$, induces an isomorphism of bimodules $\widehat{\phi}:M_f \to S_V \otimes _{\Cl_d(\R)} S_W$. 

A path in $O(V,W)$ is given by a family $f_t$ of isometries where $t\in [0,1]$. As the map $Spin(V,W) \to SO(V,W)$ is a covering map, if we pick a lift $\phi _0$ of $f_0$ as above, then we obtain a lift $\phi _t$ of $f_t$ (depending only on the homotopy class of the path). This gives the required isomorphism of bimodules
\[
\widehat \phi _1 ^{-1} \circ \widehat{\phi _0} : M_{f_0} \to M_{f_1}
\]
which doesn't depend on the choices of lift above.
\end{proof}

\newcommand{\etalchar}[1]{$^{#1}$}
\providecommand{\bysame}{\leavevmode\hbox to3em{\hrulefill}\thinspace}
\providecommand{\MR}{\relax\ifhmode\unskip\space\fi MR }
\providecommand{\MRhref}[2]{%
  \href{http://www.ams.org/mathscinet-getitem?mr=#1}{#2}
}
\providecommand{\href}[2]{#2}

\end{document}